\documentclass[final,3p,times]{./elsarticle}

\usepackage{amsmath}
\usepackage{amssymb}
\usepackage{amsthm}
\usepackage{esint}
\usepackage{mathbbol}     
\usepackage{bbm}

\numberwithin{equation}{section}

\newtheorem{theorem}{Theorem}[section]
\newtheorem{lemma}{Lemma}[section]
\newtheorem{corollary}{Corollary}[section]
\newtheorem{proposition}{Proposition}[section]
\theoremstyle{remark}
\newtheorem{remark}{Remark}[section]

\usepackage[colorlinks=true]{hyperref}
\usepackage[nameinlink]{cleveref}
\crefformat{equation}{(#2#1#3)}
\crefrangeformat{equation}{(#3#1#4)--(#5#2#6)}
\crefmultiformat{equation}{(#2#1#3)}{ and~(#2#1#3)}{, (#2#1#3)}{ and~(#2#1#3)}
\crefname{claim}{claim}{claims}

\usepackage[normalem]{ulem} 
\newcommand{\abs}[1]{\left\vert #1 \right\vert} 
\newcommand{\gradT}{\smash{\langle |\nabla T|^2 \rangle}} 
\newcommand{\mean}[1]{\left\langle #1 \right\rangle} 
\newcommand{\smallmean}[1]{\smash{\langle #1 \rangle}} 
\newcommand{\dVolume}{\mathrm{d}\vec{x}}
\renewcommand{\vec}[1]{\boldsymbol{#1}} 

\DeclareMathOperator{\Ra}{\textit{R}}
\DeclareMathOperator{\Pe}{\textit{Pe}}
\let\Pr\relax
\DeclareMathOperator{\Pr}{\textit{Pr}}

\newcommand{\Hardy}{\mathcal{H}^1}
\DeclareMathOperator{\BMO}{BMO}

\newcommand{\alert}[1]{{\color{red}#1}}

\journal{Physica D}

\begin{document}

\begin{frontmatter}



\title{Bounds on heat transfer by incompressible flows between balanced sources and sinks}

\affiliation[umich]{organization={Dept. of Mathematics, University of Michigan},
           city={Ann Arbor},
           state = {MI},
           postcode={48109},
           country={USA}}

\affiliation[icl]{organization={Dept. of Aeronautics, Imperial College London},
           city={London},
           postcode={SW7 2AZ},
           country={United Kingdom}}

\affiliation[uic]{organization={Dept. of Mathematics, Statistics, and  Computer  Science, University of Illinois Chicago},
city={Chicago},
           state = {IL},
           postcode={60607},
           country={USA}}

\author[umich]{Binglin Song}
\author[icl]{Giovanni Fantuzzi}
\author[uic]{Ian Tobasco}

\begin{abstract}
Internally heated convection involves the transfer of heat by fluid motion between a distribution of sources and sinks. Focusing on the balanced case where the total heat added by the sources matches the heat taken away by the sinks, we obtain \emph{a priori} bounds on the minimum mean thermal dissipation $\langle |\nabla T|^2\rangle$ as a measure of the inefficiency of transport. In the advective limit, our bounds scale with the inverse mean kinetic energy of the flow. The constant in this scaling law depends on the source--sink distribution, as we explain both in a pair of examples involving oscillatory or concentrated heating and cooling, and via a general  asymptotic variational principle for optimizing transport. Key to our analysis is the solution of a pure advection equation, which we do to find examples of extreme heat transfer by cellular and `pinching' flows. When the flow obeys a momentum equation, our bound is re-expressed in terms of a flux-based Rayleigh number $\Ra$ yielding $\langle |\nabla T|^2\rangle\geq C\Ra^{-\alpha}$.  The power $\alpha$ is $0, 2/3$ or $1$ depending on the arrangement of the sources and sinks relative to gravity.
\end{abstract}





\end{frontmatter}
\section{Introduction}
\noindent
Determining absolute limits on heat transport by a moving fluid is a fundamental scientific challenge. It is motivated not only by questions of planetary physics, e.g., where convection driven by radioactive decay influences plate tectonics and the generation of magnetic fields~\cite{Mulyukova2020,schubert2001mantle}, but also by the search for optimal heat exchangers~\cite{Alben2017jfm,Feppon2021}. Internally heated flows have recently attracted renewed interest after experiments and numerical simulations \cite{Lepot2018,Bouillaut2019,Kazemi2022} revealed that their heat transport can significantly exceed the known limits on  `ordinary' boundary-driven (Rayleigh--B\'enard) convection. For certain well-balanced source--sink profiles of internal heating and cooling,
the flows that set up in response to gravity appear to transport heat at a rate independent of the molecular diffusivity, achieving the `mixing-length' or `ultimate' transport scaling.
However, it remains a challenge to determine theoretically which properties of the internal heating are crucial to achieving highly efficient transport. Indeed, for an arbitrary balanced source--sink profile in an arbitrary fluid domain, it is not at all clear from the outset what transport will result. 

Thinking of the general question of assessing heat transport across a fluid domain, the first challenge is to select a globally-defined yet meaningful diagnostic measure of transport efficiency. Many known quantities that give equivalent measures for boundary-driven convection are no longer comparable for internally heated flows, and can end up following different scaling laws. 
We choose to measure heat transport using the mean-squared temperature gradient averaged over space and time. Since by Fourier's law the diffusive heat flux is controlled by the temperature gradient, it is reasonable to expect that a highly efficient transfer protocol finds a way to minimize temperature gradients overall. Other authors have studied the mean temperature~\cite{Lu2004,Goluskin2012pla,Goluskin2015ih,Whitehead2011jmp,Whitehead2012}, root-mean-squared temperature~\cite{Lepot2018,Bouillaut2019,Miquel2019} or vertical heat flux~\cite{Arslan2021flux,Arslan2022,Kumar2022ih}. For particular choices of source--sink distributions and boundary conditions such quantities are equivalent to our measure~\cite{Goluskin2016book}, but this is not generally true. 

Having selected a measure of transport, one can seek flows optimizing its value, subject to various constraints.
A tractable goal for analysis, that we pursue in this paper, is to produce \emph{a priori} bounds on transport holding for general classes of admissible flows. In particular, we shall derive a lower bound on the mean thermal dissipation of an internally heated flow, which takes into account its mean kinetic energy as well as the shape of the imposed source--sink distribution and the flow domain. We also derive a similar bound holding for buoyancy-driven internally heated convection. We work with balanced source--sink distributions that add and subtract the same amount of heat overall. The case of imbalanced heating has been studied extensively in the literature: bounds on  measures of heat transport  are known for uniform internal heating under a variety of boundary conditions~\cite{Lu2004,Goluskin2012pla,Goluskin2015ih,Whitehead2011jmp,Whitehead2012,Goluskin2016book}, as well as for essentially arbitrary source--sink distributions in a disc with a constant temperature boundary \cite{Tobasco2022}. Bounds on balanced heat transfer have also been obtained, for periodic flows under an assumption of statistical homogeneity and isotropy~\cite{Shaw2007,Thiffeault2004,Doering2006,Thiffeault2008}, and for smooth source--sink distributions that vary only in the gravity direction across a fluid layer~\cite{Miquel2019}. Here, we treat a much broader class of velocities, source--sink functions and flow domains. We also give examples illustrating the sharpness of our bounds.




There are various approaches to \emph{a priori} bounds in fluid mechanics, but for buoyancy-driven convection the relevant results can be traced back at least to the work of Malkus, Howard and Busse~\cite{Malkus1954,Howard1963,Busse1969,Howard1972,Busse1979} as well as to the   `background method' of Doering and Constantin~\cite{Doering1994pre,Constantin1995pre,Doering1996pre}. 
We follow a two-step approach that is similar to the `optimal wall-to-wall' approach of~\cite{Hassanzadeh2014,Tobasco2017,Doering2019}, and also to an approach that has been used with horizontal convection~\cite{Siggers2004,Winters2009,Rocha2019}. First, we drop the momentum equation and optimize heat transfer subject to the advection-diffusion equation alone. The resulting optimal transfer rate depends on the advective intensity, measured in a chosen norm. Then, we restore the momentum equation via a balance law relating the velocity norm to an appropriate Rayleigh number. Algebraic manipulation leads to an \emph{a priori} bound on the heat transport of momentum-conserving flows.

We turn to describe our setup and results. 
Let $\Omega \subset \mathbb{R}^d$ be a bounded Lipschitz domain in dimension $d\geq 2$, and introduce a temperature field $T(\vec{x},t)$ solving the inhomogeneous and non-dimensional advection-diffusion equation with insulating boundary conditions
\begin{equation}\label{eq:intro:heat}
        \begin{cases}
        \partial_{t}T+\vec{u}\cdot\nabla T=\Delta T+f & \text{in }\Omega,\\
        \hat{\vec{n}}\cdot\nabla T=0 & \text{at }\partial\Omega.
        \end{cases}
    \end{equation}
To ensure the source--sink function $f(\vec{x},t)$ is balanced, we set
\begin{equation*}
\fint_\Omega f(\vec{x},t) \dVolume = 0
\end{equation*}
where $\fint_\Omega \cdot \dVolume$ denotes averaging over the flow domain (per the notation in \cref{ss:notation}).
%
The advecting velocity $\vec{u}(\vec{x},t)$ obeys the divergence-free condition $\nabla\cdot\vec{u}=0$ in $\Omega$, along with the no-penetration boundary conditions $\vec{u}\cdot\hat{\vec{n}}=0$ at $\partial\Omega$.
Note $\hat{\vec{n}}$ is the outwards-pointing unit normal to the domain boundary.
These conditions on $f$ and $\vec{u}$ imply that the mean temperature $\fint_\Omega T\dVolume$ is constant in time, and we take it to be zero without loss of generality.

Given this setup, we seek bounds on the mean-squared temperature gradient
\[
\smallmean{|\nabla T|^2} := \limsup_{\tau\to\infty}\,\fint_0^\tau\fint_\Omega |\nabla T(\vec{x},t)|^2 \,\dVolume {\rm d}t
\]
where for definiteness we use the limit superior. 
In \cref{sec:nomomentum}, we derive a pair of `variational' upper and lower bounds:
\begin{equation}\label{e:intro:var-bounds}
    \left\langle
        2f\xi
        -\abs{\nabla\xi}^2
        -|\nabla\Delta^{-1}(\partial_{t}\xi
        +\vec{u}\cdot\nabla\xi)|^{2}
    \right\rangle
    \leq
    \left\langle \abs{\nabla T}^2 \right\rangle
    \leq
    \left\langle
        \abs{\nabla\eta}^2
        +|\nabla\Delta^{-1}(\partial_{t}\eta
        +\vec{u}\cdot\nabla\eta-f)|^{2}
    \right\rangle
\end{equation}
where $\xi(\vec{x},t)$ and $\eta(\vec{x},t)$ are test functions whose choice can be optimized (see \cref{thm:unsteady-VP}).
The operator $\Delta^{-1}$ is the Neumann inverse Laplacian, defined in \cref{ss:notation}.
A version of our lower bound on $\gradT$ with a steady test function $\xi(\vec{x})$ appeared in a previous paper on mixing in periodic domains \cite{Shaw2007}, along with similar bounds on $\smallmean{T^2}$ and $\smallmean{|\nabla \Delta^{-1}T|^2}$ (see also \cite{Thiffeault2004,Doering2006,Thiffeault2008} and \cite{Thiffeault2012} for a review).
To complete the picture, we allow for time-dependent test functions defined on general domains, and also provide the complementary upper bound in \cref{e:intro:var-bounds}. Moreover, we prove in \cref{th:steady-VP} that these bounds are sharp in the steady case $f=f(\vec{x})$ and $\vec{u}=\vec{u}(\vec{x})$, meaning that an optimal choice of test function evaluates $\gradT$. This is the analog of results obtained in~\cite{Tobasco2017,Doering2019,Souza2020} for boundary-driven flows and in~\cite{Tobasco2022} for internally-heated flows with cooled boundaries. Rather than repeat the `symmetrization argument' from these papers,  we present a different and likely more flexible proof in which the test functions  are introduced as Lagrange multipliers for enforcing the advection-diffusion equation (much like the argument in \cite{Shaw2007}).


\Cref{sec:univ-lower-bounds,sec:asymptotics} go on to ask what the variational bounds \cref{e:intro:var-bounds} imply for optimal flows. 
By optimal, we mean flows that minimize $\gradT$ subject to a constraint on the flow intensity, such as might be given by fixing the value of the mean kinetic energy $\langle |\vec{u}|^2\rangle$. To ease the presentation, we treat steady source--sink functions $f(\vec{x})$ that are not identically zero from here on; we expect our results can be extended to unsteady sources/sinks, and we remark on this below.
\Cref{sec:univ-lower-bounds} starts by bounding $\gradT$ in terms of the mean kinetic energy: we show that there are positive constants $C_1$, $C_2$ and $C_3$ depending on the domain $\Omega$, the dimension $d$ and the source--sink distribution $f$ such that
\begin{equation}\label{e:intro:lb-ke}
    \langle |\nabla T|^2 \rangle \geq \frac{C_1}{C_2+ C_3 \langle |\vec{u}|^2 \rangle }.
\end{equation}
This follows from \cref{th:lb-BMO} and the subsequent discussion on how to select $\xi$. A similar bound was proved in \cite{Shaw2007, Doering2006, Thiffeault2012} for periodic and statistically homogeneous, isotropic flows. 
We too think of the ratios $C_1/C_2$ and $C_1/C_3$ as (squared) norms that detect the structure of the sources and sinks. In the conductive limit $\langle |\vec{u}|^2\rangle \to 0$, the first ratio dominates the bound and can be chosen, as usual, to involve a negative Sobolev norm (see \cref{e:lb:xi-Hm1}). In the advective limit $\langle |\vec{u}|^2\rangle \to \infty$ the second ratio is more important; we explain how to choose it based on a Hardy space norm (see \cref{e:lb:xi-Hardy}). While this norm is not completely unlike the $L^1$-norm, it produces a strictly better bound for problems with point-like sources and sinks. See \cref{ss:BMO-bound} for its definition, and \cref{ss:lower-bounds-BMO-hardy} for the proof of \cref{e:intro:lb-ke}.

\begin{figure}
    \centering
    \includegraphics[width=.9\textwidth]{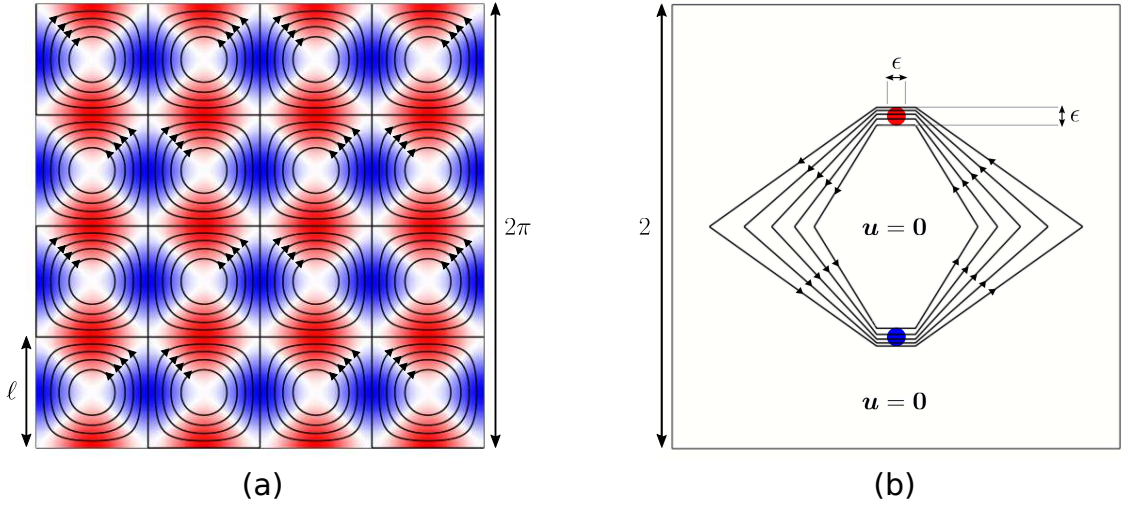}
       \caption{Examples of nearly optimal flows. Heat sources (red) and sinks (blue) oscillate sinusoidally on a scale $\sim \ell$ in (a) and are concentrated in regions of size $\sim\epsilon$ in (b). Black lines with arrows show  streamlines of the cellular flow achieving $\langle |\nabla T|^2 \rangle\sim \Pe^{-2}$ in (a) and of the `pinching' flow achieving $\langle |\nabla T|^2\rangle \sim \log^2(\epsilon^{-1})\Pe^{-2}$ in (b). The squared P\'eclet number $\Pe^2$ sets the kinetic energy of the flows.
    \label{fig:flows}}
\end{figure}

To help clarify our lower bound, and to explain what it takes to find optimal (or at least nearly optimal) flows, we study a pair of examples in  \cref{ss:lower-bounds-example} involving sinusoidal heating and cooling or approximate point sources and sinks. \Cref{fig:flows} illustrates the setups we have in mind: the heating/cooling in panel (a) varies sinusoidally on a scale $\sim l$; the point-like sources and sinks in panel (b) are concentrated in regions of size $\sim\epsilon$ (the figure shows discs for simplicity). Given these setups, we prove upper and lower bounds on the minimum mean-square temperature gradient that match in terms of their scalings with respect to each example's parameters. Precisely, we show that
\begin{equation}\label{e:intro:examples}
\min_{\substack{\vec{u}(\vec{x},t) \\ \smallmean{|\vec{u}|^2} \leq \Pe^2 \\ \partial_t T + \vec{u}\cdot\nabla T = \Delta T + f}} \gradT \sim
\begin{cases}
    \min\left\{\ell^2,\frac{1}{\Pe^2} \right\}
    &\text{for sinusoidal heating (\cref{th:cellular-flows})}
    \\[1ex]
    \min
    \left\{
        \log\frac1\epsilon,\,
        \left(\log\frac1\epsilon\right)^2\frac{1}{\Pe^2}
        \right\}
    &\text{for concentrated heating (\cref{th:pinching})}
\end{cases}
\end{equation}
with prefactors that are independent of all parameters. (The notation $X\sim Y$ means that there are numerical constants $C,C'>0$ such that $CY \leq X \leq C'Y$). Parsing  \cref{e:intro:lb-ke} in each example produces the lower bound half of \cref{e:intro:examples}. For the matching upper bounds, we construct the flows illustrated in \cref{fig:flows} and select test functions $\eta$ to estimate their transport via \cref{e:intro:var-bounds}. We use a cellular flow structure for sinusoidal heating and a pinching effect for approximate point sources and sinks. It was the analysis of pinching flows that led us to the Hardy space norm; other, more familiar norms gave strictly sub-optimal bounds. 


The search for flows optimizing heat transfer is an active area of research; see \cite{Tobasco2022,Marcotte2017,Iyer2022} for flows enhancing heat transport with imbalanced heating, and \cite{Iyer2010} for flows inhibiting heat transport in non-disc domains. 
In this paper, we base our constructions on an ability to solve the pure and steady advection system
\begin{equation}\label{e:intro:pure-advection}
    \begin{cases}
        \vec{u}_0\cdot\nabla T_0=f & \text{in }\Omega\\
        \nabla \cdot \vec{u}_0 =0 & \text{in }\Omega\\
        \vec{u}_0\cdot\hat{\vec{n}} = 0& \text{at }\partial\Omega
     \end{cases}.
 \end{equation}
Though understanding precisely when this system has a solution is a difficult and open problem (see \cite{Lindberg2017_Hardy} for a recent account), it is not so difficult to show its relevance for optimizing heat transfer in the advective limit. \Cref{sec:asymptotics} studies this limit in detail, and achieves the following conditional result: if \cref{e:intro:pure-advection} admits a (regular-enough) solution, then
\begin{equation*}
    \lim_{\Pe\to\infty}
    \min_{\substack{\vec{u}(\vec{x}) \\ \|\vec{u}\|_X\leq\Pe \\ \vec{u}\cdot\nabla T = \Delta T + f}}
    \, \Pe^2 \fint_\Omega |\nabla T|^2\dVolume
    =
    \min_{\substack{\vec{u}_0(\vec{x}),T_0(\vec{x}) \\ \|\vec{u}_0\|_{X} \leq 1\\
    \vec{u}_0\cdot\nabla T_0=f
    }
    }\, \fint_\Omega |\nabla T_0|^2\dVolume.
\end{equation*}
In the right-hand problem, the minimization is over all solutions of \cref{e:intro:pure-advection}. After a rescaling, its optimizers give the limit points of almost-minimizing sequences on the left; see \cref{th:limitbd} for the precise statement.
Here, we allow for general families of steady velocities belonging to a Banach space $X$ that is continuously embedded into $L^d(\Omega)$, amongst other requirements. The dimensionless parameter $\Pe=U^2L/\kappa$ is the P\'eclet number, where $U$ is a characteristic velocity scale, $L$ is a characteristic lengthscale and $\kappa$ is the thermal diffusivity. The result captures the intuition that optimal velocities find a way to minimize thermal dissipation while achieving (essentially) perfect advection, and shows how to compute the optimal prefactor in the scaling law $\min\, \gradT \sim C(\Omega,d,f) \Pe^{-2}$. 

Finally, in \cref{sec:buoyancy} we bound the heat transfer of momentum-constrained flows driven by a steady balanced source--sink function $f(\vec{x})$ and a steady conservative gravitational acceleration $\vec{g}(\vec{x})$. In addition to the advection-diffusion equation \cref{eq:intro:heat} for the temperature $T$, we let the velocity $\vec{u}$ solve the non-dimensional Boussinesq equation
%
\begin{equation*}
	\Pr^{-1} \left( \partial_t \vec{u} + \vec{u}\cdot \nabla \vec{u} \right) = \Delta \vec{u} + \Ra T \vec{g} - \nabla p
\end{equation*}
with $\vec{g}=\nabla\varphi$. The choice $\varphi=z$ yields the usual Boussinesq equations with gravity in the negative $z$-direction, and we allow for other choices as well.
The non-dimensional parameters $\Pr$ and $\Ra$ are the usual Prandtl number and a `flux-based' Rayleigh number \cite{Goluskin2016book} measuring the strength of the sources and sinks relative to diffusion: $\Pr = \nu/\kappa$ and $\Ra = \alpha L^5 Q/(\nu \kappa^2)$.
Again, $L$ is a characteristic lengthscale and $\kappa$ is the thermal diffusivity; also, $\nu$ is the kinematic viscosity,  $\alpha$ is the thermal expansion coefficient and $Q$ is a characteristic heating and cooling rate per unit volume (which sets the dimensional amplitude of $f$).
After deriving a set of basic balance laws, we relate the mean enstrophy and energy of the flow to the Rayleigh number, and thereby obtain a trio of Rayleigh-dependent lower bounds (\cref{th:mom:bounds}).
These bounds are in the general form
\begin{equation*}
    \gradT \geq  \frac{C(\Omega,d,f,\vec{g})}{\Ra^{\alpha}}
\end{equation*}
with $\alpha = 0$, $2/3$ or $1$ depending on the sign of $\langle f\varphi  \rangle$ and for large enough $R$; we give intuition for this  below. 

 \Cref{sec:discussion} is a conclusion section that includes a discussion of open questions and future directions of research. 



\subsection{Notation}\label{ss:notation}
\noindent
Here we summarize some common notational conventions. We use $X \vee Y$ and $X\wedge Y$ for the maximum and minimum of two quantities. 
We write $X\lesssim Y$ if there is a constant $C$ with $X \leq C Y$, and  $X\sim Y$ if $X\lesssim Y \lesssim X$. If the constant $C$ depends on a parameter $a$, we indicate this by writing $X \lesssim_a Y$. The notation $X \ll Y$ means that $X/Y \to 0$ in a limit. Likewise, $o(X)$ is a quantity tending to zero upon division by $X$.

The $d$-dimensional volume of a set $A$ is $|A|$. The average of $\varphi$ over $A$ is then
\[
\fint_{A} \varphi \dVolume = \frac{1}{|A|}\int_A\varphi\dVolume.
\]
The notations $\langle \varphi \rangle_\tau$ and $\mean{\varphi}$ give spatial-temporal averages over $\Omega$ and up to time $\tau$, or across infinite time, respectively:
\begin{equation*}
	\mean{\varphi}_\tau = \fint_0^\tau \fint_\Omega \varphi(\vec{x},t) \dVolume {\rm d}t\quad\text{and}\quad \mean{\varphi} = \limsup_{\tau \to \infty}\, \mean{\varphi}_\tau.
\end{equation*}
We only use the limit superior long-time average.

As usual, $L^p(\Omega)$ is the space of functions whose $p$-th power is integrable on $\Omega$. We write $H^1(\Omega)$ for the Sobolev space of functions in $L^2(\Omega)$ whose weak derivatives are in $L^2(\Omega)$. We use $H^{-1}(\Omega)$ for the space of continuous linear functionals on $H^1(\Omega)$ that are mean-free, meaning that they take constant functions to zero, and define the dual norm
\begin{equation*}\label{e:Hm1-norm}
    \|g\|_{H^{-1}(\Omega)} := \max_{ \varphi(\vec{x})}\,
    \frac{\int_\Omega g \varphi\dVolume}{\left(\int_\Omega |\nabla\varphi|^2\dVolume\right)^\frac12}  = \left(\int_\Omega \abs{\nabla \Delta^{-1} g}^2 \dVolume \right)^\frac12.
\end{equation*}
This and other such maximizations are performed over non-constant $\varphi\in H^1(\Omega)$. Not every $g\in H^{-1}(\Omega)$ is a function, in which case the `integral' $\int_\Omega g \cdot \dVolume$ stands for the action of $g$ as a functional. 
Hardy and BMO spaces will be used; see \cref{ss:BMO-bound} for definitions and a brief review.

Finally, since this paper deals only with insulating temperature boundary conditions, we write $\Delta^{-1}$ for the inverse Laplacian operator with Neumann boundary conditions. In formulas, $h=\Delta^{-1}g$ if $\fint_\Omega h \dVolume = 0$ and
\begin{equation*}
    \begin{cases}
    \Delta h=g & \text{in }\Omega\\
    \hat{\vec{n}}\cdot\nabla h=0 & \text{at }\partial\Omega
    \end{cases}.
\end{equation*}

\subsection{Acknowledgements}
\noindent
We thank David Goluskin and John Craske for insightful conversations about the physics of internally heated flows, and dedicate this article to the late Professor Charlie Doering, who opened the door for us to the world of mathematical fluid dynamics.
G.F.\ was supported by an Imperial College Research Fellowship and thanks the Isaac Newton Institute for Mathematical Sciences, Cambridge for support and hospitality during the programme “Mathematical aspects of turbulence: where do we stand?” (EPSRC grant number EP/R014604/1) where work on this paper was undertaken.
I.T.\ was supported by National Science Foundation Award DMS-2025000.

\section{\alert{Variational} bounds on heat transfer in an insulated domain}
\label{sec:nomomentum}
\noindent
\Cref{ss:s2-unsteady-bounds} derives upper and lower bounds on the mean square thermal dissipation $\gradT$ of a general unsteady incompressible flow $\vec{u}(\vec{x},t)$ and a general unsteady and balanced source--sink distribution $f(\vec{x},t)$.
These bounds involve a pair of test functions,
which can be optimized based on the details of $\vec{u}$ and $f$.
In the steady case where $\vec{u}=\vec{u}(\vec{x})$ and $f=f(\vec{x})$, the optimization evaluates $\langle |\nabla T|^2\rangle$ so that the bounds are sharp (see \cref{ss:s2-steady-bounds}).
One can also choose the test functions to bound $\gradT$ in terms of a bulk measure of the flow intensity, such as the mean kinetic energy $\smallmean{|\vec{u}|^2}$; we do so starting in \cref{sec:univ-lower-bounds}. 

\subsection{Variational bounds for unsteady flows and source--sink distributions}
\label{ss:s2-unsteady-bounds}
\noindent
Define the admissible set
\begin{equation}
\label{e:test-function-space}
    \mathcal{A}=\left\{ \theta\in L_{\text{loc}}^{2}\left(0,\infty;H^{1}(\Omega)\right):
    \;\partial_{t}\theta\in L_{\text{loc}}^{2}\left(0,\infty;H^{-1}(\Omega)\right),\; \theta(\cdot ,\tau)\in L^\infty(\Omega)\text{ a.e. } \tau,
    \;\lim_{\tau\rightarrow\infty}\frac{1}{\sqrt{\tau}}\|\theta(\cdot,\tau)\|_{L^{2}(\Omega)}
    =0
    \right\}.
\end{equation}
Let $\lambda_1$ be the first non-trivial Neumann eigenvalue of the (negative) Laplacian on $\Omega$. It is the largest constant such that $\lambda_1  \|\theta\|^2_{L^2(\Omega)} \leq \|\nabla \theta\|^2_{L^2(\Omega)}$ for all mean-zero $\theta \in H^1(\Omega)$.

\begin{theorem}
\label{thm:unsteady-VP} Let $\vec{u}(\vec{x},t)$ be divergence-free with $\vec{u} \cdot \hat{\vec{n}} = 0$ at $\partial\Omega$ and $\langle |\vec{u}|^2\rangle <\infty$, and let $f(\vec{x},t)$ satisfy $\fint_\Omega f(\vec{x} ,t)\dVolume = 0$ for $t>0$. Assume there exists $a_0,a_1\in(0,1)$ and $p\in[d/2,\infty]$ such that
    \begin{equation}\label{eq:assumptions-on-f}
        \lim_{\tau\rightarrow\infty}\frac{1}{\tau}\int_{0}^{\tau}e^{-2a_0\lambda_{1}\left( \tau-t \right)}\|f(\cdot,t)\|_{H^{-1}(\Omega)}^2 \, {\rm d}t=0
        \quad\text{and}\quad  \int_{0}^{\tau}\frac{e^{-a_1\lambda_{1}(\tau -t)}}{(\tau-t)^{\frac{d}{2p}}}\|f(\cdot,t)\|_{L^{p}(\Omega)}\,dt<\infty\quad\forall\,\tau>0.
    \end{equation}
Given any weak solution $T(\vec{x},t)$ of
    \begin{equation}\label{eq:adv-diff_weak_forproof}
        \begin{cases}
        \partial_{t}T+\vec{u}\cdot\nabla T=\Delta T+f & \text{in }\Omega\\
        \hat{\vec{n}}\cdot\nabla T=0 & \text{at }\partial\Omega
        \end{cases}
    \end{equation}
   with $T(\cdot ,0)\in L^2(\Omega)$, the upper and lower bounds
    \begin{equation}\label{eq:vp}
        \left\langle
            2f\xi
            -\abs{\nabla\xi}^2
            -|\nabla\Delta^{-1}(\partial_{t}\xi
            +\vec{u}\cdot\nabla\xi)|^{2}
        \right\rangle
        \leq
        \left\langle \abs{\nabla T}^2 \right\rangle
        \leq
        \left\langle
            \abs{\nabla\eta}^2
            +|\nabla\Delta^{-1}(\partial_{t}\eta
            +\vec{u}\cdot\nabla\eta-f)|^{2}
        \right\rangle
    \end{equation}
   hold for all $\eta,\xi\in \mathcal{A}$.
\end{theorem}

\begin{remark}\label{rem:f-assumptions}
 The two assumptions on $f$ in \cref{eq:assumptions-on-f} play different roles. 
 The first one ensures that $\|T(\cdot,\tau)\|_{L^2(\Omega)} \ll \sqrt{\tau}$ as $\tau\to \infty$, so that the integration-by-parts identity   $\langle \abs{\nabla T}^2\rangle=\left\langle fT\right\rangle$ holds. The second one implies that $T(\cdot,\tau)\in L^\infty(\Omega)$ for a.e.\ $\tau>0$. This allows us to define the notion of weak solutions in the usual way, by testing the given equation against functions in $H^1(\Omega)$ and integrating by parts (following, e.g., \cite{EvansPDE}); it also ensures that $\int_\Omega \vec{u}\cdot\nabla T T\dVolume = 0$. As these remarks are more or less standard fare in parabolic regularity theory \cite{EvansPDE, Lieberman1996}, we shall not present their proofs, but instead point to the notes \cite{Kiselev2021} for an exposition that is readily adapted to our setting. In brief, the desired $L^2$-bound follows from Gronwall's inequality by testing the equation against $T$; the $L^\infty$-bound follows from known `heat kernel' bounds on the forward-time solution map of the homogeneous equation (bounding it from 
$L^p(\Omega)$ to $L^\infty(\Omega)$ by a multiple of $e^{-a\lambda_1\Delta t}(\Delta t)^{-d/2p}$ for any $a\in(0,1)$, across a time increment $\Delta t$). 
\end{remark}
\begin{remark}\label{rem:extra-integrability-velocity}One can do away with the second assumption on $f$ in \cref{eq:assumptions-on-f} by imposing enough integrability on $\vec{u}$ to the point that $\vec{u}\cdot\nabla \theta \in H^{-1}(\Omega)$ a.e.\ in time, for any $\theta \in H^1(\Omega)$. Since by the Sobolev embedding theorem \cite{EvansPDE} $L^{2d/(d+2)}(\Omega)$ is included in $H^{-1}(\Omega)$ when $d > 2$, it suffices to assume that $\vec{u}\in L^{d}(\Omega)$ a.e.\ in time. The $L^\infty$ assumption on the test functions in $\mathcal{A}$ can then be removed. These statements continue to hold in the borderline case $d=2$ due to `div-curl' character of the product $\vec{u}\cdot\nabla \theta$; see \cref{ss:BMO-bound} or \cite{Tobasco2022} for more details.
\end{remark}
\begin{proof}[Proof of \cref{thm:unsteady-VP}]


We apply the method of Lagrange multipliers, with the advection-diffusion equation  \cref{eq:adv-diff_weak_forproof} as the constraint and the test functions $\eta$ and $\xi$ as multipliers.  
Let  $\mean{\cdot}_\tau = \fint_0^\tau \fint_\Omega \cdot \, \dVolume {\rm d}t$ be the average of a quantity over $\Omega\times (0,\tau)$. 
To prove the lower bound in \cref{eq:vp}, start with the weak form of the advection-diffusion equation, which states that
\begin{equation}\label{eq:adv-diff-weakform}
    \mean{f \xi - \nabla T \cdot  \nabla \xi - \partial_t T \xi - \vec{u}\cdot\nabla T \xi}_\tau = 0
\end{equation}
for any  $\xi \in \mathcal{A}$ and  $\tau>0$.
Thus,  
    \begin{align}
        \mean{ \abs{\nabla T }^{2}}_{\tau} 
        &=
        \left\langle \abs{\nabla T}^{2} + 2\left(f\xi - \nabla T \cdot \nabla \xi- \partial_{t}T \xi -\vec{u}\cdot\nabla T \xi \right)
        \right\rangle _{\tau}
        \nonumber\\
        &=
        \left\langle
        \abs{\nabla T}^{2}+2 f\xi
        -2\nabla T \cdot \nabla \xi
        +2 T \left(\partial_{t}\xi + \vec{u}\cdot\nabla\xi \right)\right\rangle_{\tau}
        -\left.\frac{2}{\tau}\fint_{\Omega} T\xi \dVolume\right\vert_{t=0}^{t=\tau}
        \nonumber\\
        &\geq
        \inf_{\theta}\,
        \left\langle
        \abs{\nabla \theta}^{2} +2f\xi
        -2\nabla \theta \cdot \nabla\xi
        +2 \theta \left(\partial_{t}\xi + \vec{u}\cdot\nabla\xi \right)\right\rangle_{\tau}
        -
        o_{\tau}(1) 
        \label{e:opt:lb-derivation}
    \end{align}
    where $o_\tau(1)$ denotes a term that goes to zero as $\tau \to \infty$.
To see this last step, note that $\|T(\cdot,\tau)\|_{L^2(\Omega)} \ll \sqrt{\tau}$ as explained in \cref{rem:f-assumptions}. Also,
    $\|\xi(\cdot,\tau)\|_{L^2(\Omega)} \ll \sqrt{\tau}$ by the definition of the admissible set $\mathcal{A}$ in \cref{e:test-function-space}. Hence,
    \begin{equation*}
        \abs{  \frac{1}{\tau}
        \left.
        \fint_{\Omega}
        T\xi \dVolume
        \right\vert_{t=0}^{t=\tau}}
        \leq
        \frac{1}{\tau \abs{\Omega}}
        \|T(\cdot,\tau)\|_{L^2(\Omega)}
        \|\xi(\cdot,\tau)\|_{L^2(\Omega)}
        + \abs{\frac{1}{\tau}
        \fint_{\Omega}
        T(\vec{x},0)\xi(\vec{x},0)
        \dVolume } \to 0
        \quad \text{as } \tau \to \infty.
    \end{equation*}


To evaluate the infimum in \cref{e:opt:lb-derivation}, we make use of its Euler--Lagrange equation
\begin{equation*}
    \begin{cases}
        \Delta \theta = \Delta \xi + \partial_t \xi + \vec{u}\cdot \nabla \xi &\text{on }\Omega\\
        \hat{\vec{n}} \cdot \nabla \theta = 0 &\text{at }\partial\Omega
    \end{cases}
\end{equation*}
which gives the optimal $\theta$ at each fixed time. Testing against $\theta$ and integrating by parts shows that
\[
\mean{
    \abs{\nabla \theta}^{2}}_{\tau} = \mean{\nabla \theta \cdot \nabla\xi
    - \theta \left(\partial_{t}\xi + \vec{u}\cdot\nabla\xi \right)}_{\tau}.
\]
Therefore, by \cref{e:opt:lb-derivation},
    \begin{align}
        \mean{ \abs{\nabla T }^{2}}_{\tau}
        &\geq
        \left\langle
        2f\xi  -
        \abs{
            \nabla\Delta^{-1}\left( \Delta \xi +
            \partial_{t}\xi+\vec{u}\cdot\nabla\xi\right)
            }^{2}
        \right\rangle_{\tau}
        -o_\tau(1) \nonumber
        \\
        &=
        \left\langle
        2f \xi  -
        \abs{\nabla \xi}^2 -
        \abs{\nabla\Delta^{-1}
            (\partial_{t}\xi+\vec{u}\cdot\nabla\xi)
            }^{2}
        \right\rangle_{\tau}
        -o_\tau(1).\label{e:opt:lb-derivation-step2}
    \end{align}
Note the cross term vanishes since
    \begin{equation*}
        2\left\langle
        \nabla\xi \cdot
        \nabla\Delta^{-1}
        \left(
        \partial_{t}\xi+\vec{u}\cdot\nabla\xi \right)
        \right\rangle_{\tau}
        =
        -2\left\langle
        \xi \left(
        \partial_{t}\xi+\vec{u}\cdot\nabla\xi\right)
        \right\rangle _{\tau}
        =
        -\frac{1}{\tau} \int_0^\tau \frac{\rm d}{{\rm d}t} \|\xi\|_{L^2(\Omega)}^2 \, {\rm d}t
        =
        o_\tau(1)
    \end{equation*}
by the growth conditions on $\xi$ and 
the no-penetration conditions for $\vec{u}$. Taking $\tau \to \infty$ in \cref{e:opt:lb-derivation-step2} yields the lower bound.

The upper bound  in \cref{eq:vp} is proved by a similar argument. The key is to find a version of
$\gradT$ that lends itself to maximization, rather than minimization.  Multiplying the advection-diffusion equation \cref{eq:adv-diff_weak_forproof} by $T$ and integrating by parts shows that
\begin{equation*}
    \langle \abs{\nabla T}^2\rangle_\tau
    =\left\langle fT\right\rangle_\tau + o_\tau(1),
\end{equation*}
as in the first part of \cref{rem:f-assumptions}. This allows us to rewrite
\[
\gradT_\tau = \smallmean{2fT- \abs{\nabla T}^2}_\tau + o_\tau(1)
\]
and mimic the previous argument, but with an upper bound. Specifically, using a Lagrange multiplier $\eta \in \mathcal{A}$ in the weak form \cref{eq:adv-diff-weakform} of the advection--diffusion equation, we can write that
    \begin{align*}
        \left\langle 2fT - \abs{\nabla T}^2 \right\rangle _{\tau} 
        &=
        \left\langle
        2fT-\abs{\nabla T}^2
        -2\left(f\eta
        - \nabla T \cdot \nabla \eta
        -\partial_{t}T \eta - \vec{u}\cdot\nabla T \eta
            \right)\right\rangle_{\tau}
        \\
        &\leq \sup_{\theta}\,
        \left\langle
        2f\theta-\abs{\nabla \theta}^2
        -2f\eta
        +2 \nabla \theta \cdot \nabla \eta
        -2\theta \left(\partial_{t} \eta + \vec{u}\cdot\nabla \eta
            \right)\right\rangle_{\tau}
        +o_{\tau}(1)
        \\
        &=
        \left\langle
        -2f\eta
        +\abs{
        \nabla\Delta^{-1}\left(\Delta \eta +
        \partial_{t}\eta+\vec{u}\cdot\nabla\eta-f\right)}^{2}\right\rangle _{\tau}+o_{\tau}(1)
        \\
        &=
        \left\langle \abs{\nabla\eta}^{2}
        +
        \abs{\nabla\Delta^{-1}(\partial_{t}\eta+\vec{u}\cdot\nabla\eta-f)}^{2}
        \right\rangle_{\tau}
        +o_{\tau}(1).
    \end{align*}
To pass between the first and second lines, integrate by parts to find only $o_\tau(1)$ contributions. Then, optimize over $\theta$ like before. Taking $\tau \to \infty$ gives the upper bound.
\end{proof}

\subsection{Sharpness in the steady case}
\label{ss:s2-steady-bounds}
\noindent
The variational bounds in \cref{thm:unsteady-VP} apply to both steady and unsteady $\vec{u}$ and $f$. In the steady case, these bounds cannot be improved. We adapt the argument from~\cite{Tobasco2017,Doering2019}.

\begin{corollary}\label{th:steady-VP}
    Let $\vec{u}(\vec{x})$, $f(\vec{x})$ and $T(\vec{x},t)$ be as in \cref{thm:unsteady-VP} (in particular let $f\in L^{p}(\Omega)$ for some $p>d/2$). Then,
    \[
        \max_{\xi \in H^1(\Omega)\cap L^\infty(\Omega)}
        \,\fint_{\Omega}
            2f\xi
            -\abs{\nabla\xi}^2
            -|\nabla\Delta^{-1}(\vec{u}\cdot\nabla\xi)|^{2}
            \,\dVolume
        =
        \left\langle \abs{\nabla T}^2\right\rangle
        =
        \min_{\eta \in H^1(\Omega)\cap L^\infty(\Omega)}
        \,\fint_{\Omega}
            \abs{\nabla\eta}^2
            +|\nabla\Delta^{-1}(\vec{u}\cdot\nabla\eta-f)|^{2}
            \,\dVolume.
    \]
\end{corollary}
\begin{remark}
The integrability assumption on $f$ derives from the second part of~\cref{eq:assumptions-on-f}, which guarantees for steady $T(\vec{x})$ that $\|T\|_{L^\infty(\Omega)}\lesssim_{\Omega,d,p}\|f\|_{L^p(\Omega)}<\infty$ \cite{Kiselev2021}. The first part of~\cref{eq:assumptions-on-f} is redundant by Sobolev embedding \cite{EvansPDE}. 
\end{remark}
\begin{remark} \label{rem:reduced-regularity} Following up from \cref{rem:extra-integrability-velocity}, if $\vec{u}\in L^d(\Omega)$ then the result holds for $f\in H^{-1}(\Omega)$ and with test functions $\xi,\eta \in H^1(\Omega)$.
\end{remark}
\begin{proof}
    Optimizing the upper and lower bounds in \cref{thm:unsteady-VP} over all steady fields $\xi,\eta \in \mathcal{A}$ gives that
    \begin{equation}\label{e:steady-var-princ-initial}
        \sup_{\xi\in H^1(\Omega)\cap L^\infty(\Omega)}
        \,\fint_{\Omega}
            2f\xi
            -\abs{\nabla\xi}^2
            -|\nabla\Delta^{-1}(\vec{u}\cdot\nabla\xi)|^{2}
            \,\dVolume
        \leq
        \left\langle \abs{\nabla T}^2\right\rangle
        \leq
        \inf_{\eta \in H^1(\Omega)\cap L^\infty(\Omega)}
        \,\fint_{\Omega}
            \abs{\nabla\eta}^2
            +|\nabla\Delta^{-1}(\vec{u}\cdot\nabla\eta-f)|^{2}
            \,\dVolume.
    \end{equation}
    There is no loss of generality in taking $\xi$ and $\eta$ to be mean-free. The resulting variational problems are respectively strictly concave and strictly convex, so that solving them is the same as solving their  Euler--Lagrange equations:
    %
    \begin{align*}
        \Delta\xi
        &=\vec{u}\cdot\nabla\Delta^{-1}\left(\vec{u}\cdot\nabla\xi\right)-f,\\
        \Delta\eta
        &=\vec{u}\cdot\nabla\Delta^{-1}\left(\vec{u}\cdot\nabla\eta-f\right)
    \end{align*}
    %
    with $\hat{\vec{n}}\cdot\nabla\xi = \hat{\vec{n}}\cdot\nabla\eta= 0$ at $\partial\Omega$.
    Equivalently, we must solve
    \begin{equation}\label{e:steady:el-alternative}
        \left\{
        \begin{aligned}
        \vec{u}\cdot\nabla\eta-f & =\Delta\xi\\
        \vec{u}\cdot\nabla\xi & =\Delta\eta
        \end{aligned}
        \right.
    \end{equation}
    with the same boundary conditions. The change of variables $T =  \xi + \eta$ and $T_{\rm adj} = \xi - \eta$ recovers the steady heat equation and its adjoint:
    \begin{align*}
        \vec{u} \cdot \nabla T &= \Delta T + f,\\
        -\vec{u} \cdot \nabla T_{\rm adj} &= \Delta T_{\rm adj} + f
    \end{align*}
    with $\hat{\vec{n}}\cdot\nabla T = \hat{\vec{n}}\cdot\nabla T_{\rm adj}= 0$ at $\partial\Omega$.
    These last equations define $T$ and $T_{\rm adj}$ and ensure their boundedness, due to the second part of our assumption~\cref{eq:assumptions-on-f} on $f$ and \cref{rem:f-assumptions}. Hence, $\xi := (T+T_{\rm adj})/2$ and $\eta := (T-T_{\rm adj})/2$ are admissible in \cref{e:steady-var-princ-initial}, and we can proceed to evaluate their bounds.

   First, note that
    \begin{equation}\label{e:gradT-steady}
        \left\langle \abs{\nabla T}^2\right\rangle =
        \fint_\Omega
        \abs{\nabla\xi}^2+\abs{\nabla\eta}^2
        \,\dVolume
    \end{equation}
    because testing the second equation in \cref{e:steady:el-alternative} against $\xi$ yields
    $\fint_\Omega{\nabla\eta\cdot\nabla\xi}\dVolume=
    -\fint_\Omega{\vec{u}\cdot\nabla\xi \xi}\dVolume=0$.
    Now, substitute $\xi = \Delta^{-1}(\vec{u}\cdot \nabla \eta - f)$ into the right-hand side of~\cref{e:gradT-steady} to obtain
    \begin{equation*}
        \mean{\abs{\nabla T}^2} =
        \fint_\Omega \abs{\nabla\eta}^2 + \abs{\nabla\Delta^{-1}(\vec{u}\cdot \nabla \eta - f) }^2
        \,\dVolume.
    \end{equation*}
    This verifies the optimality of $\eta$ and proves the second half of \cref{th:steady-VP}. To prove the first half, note the identity
\begin{equation}\label{eq:used-identity}
    \fint_\Omega{\abs{\nabla\eta}^2}\dVolume=\fint_\Omega{f\xi - \abs{\nabla\xi}^2}\dVolume
 \end{equation}
which derives from testing the first equation in \cref{e:steady:el-alternative} by $\xi$ and the second by $\eta$ and combining the results. Indeed,
    \begin{equation*}
        \fint_\Omega
        {f\xi - \abs{\nabla \xi}^2 - \abs{\nabla\eta}^2}
        \,\dVolume
        =
        \fint_\Omega
        f\xi -\abs{ \nabla \xi}^2 + (\vec{u}\cdot\nabla \xi) \eta
        \,\dVolume
        =
        \fint_\Omega
        f\xi -\abs{ \nabla \xi}^2 - (\vec{u}\cdot\nabla \eta) \xi
        \,\dVolume
        =0.
    \end{equation*}
Combining \cref{e:gradT-steady} and \cref{eq:used-identity} and using that  $\eta = \Delta^{-1}(\vec{u} \cdot \nabla \xi)$ we conclude that
    \begin{equation*}
        \mean{\abs{\nabla T}^2}
        =
        \fint_\Omega
        { 2\abs{\nabla\eta}^2 + \abs{\nabla\xi}^2  - \abs{\nabla\eta}^2 }
        \,\dVolume
        =
        \fint_\Omega
        { 2f\xi - \abs{\nabla\xi}^2 - \abs{\nabla\Delta^{-1} \vec{u}\cdot\nabla\xi}^2 }
        \,\dVolume
    \end{equation*}
   as required.%
 \end{proof}



\section{Bounds on energy-constrained flows}
\label{sec:univ-lower-bounds}
\noindent
The previous section achieved upper and lower bounds on $\gradT$ in terms of a pair of test functions, the choice of which was left up to the reader depending on the application. We now demonstrate how knowledge of the mean kinetic energy $\langle |\vec{u}|^2\rangle$ along with the structure of the source--sink distribution $f$ can be used to achieve the lower bound
\begin{equation}
    \label{eq:lb-generic}
    \langle |\nabla T|^2 \rangle \geq \frac{C_1}{C_2+ C_3 \langle |\vec{u}|^2 \rangle }.
\end{equation}
%
We base our approach on a well-known inequality of Coifman, Lions, Meyers and Semmes \cite{Coifman1990}, which we introduce in \cref{ss:BMO-bound} along with the requisite functional analysis involving Hardy and BMO spaces. This inequality explains how the advection term $\vec{u}\cdot\nabla T$ inherits additional regularity beyond a typical dot product from the fact that it involves divergence- and curl-free fields. Using it, we achieve \cref{eq:lb-generic} in \cref{ss:lower-bounds-BMO-hardy}.

\Cref{ss:lower-bounds-example} goes on to discuss a pair of examples where our methods establish the scaling law of $\min\, \gradT$ with respect to $\langle |\vec{u}|^2\rangle$ and certain features of $f$. In each example, we apply \cref{th:lb-BMO} with a suitable test function to deduce a lower bound. We then saturate the scaling behaviors of this bound by constructing nearly optimal velocity fields. Part of the puzzle is to understand when flowing is significantly better than not, and indeed this is reflected by a cross-over in the optimal scaling laws achieved in \cref{th:cellular-flows,th:pinching}.

\subsection{A brief introduction to \texorpdfstring{$\Hardy$}{H^1} and \texorpdfstring{$\BMO$}{BMO}}
\label{ss:BMO-bound}
\noindent
First, we introduce the functional analytic framework we use to prove our kinetic energy-based  bounds. We leave out most of the proofs, and point to the  references~\cite{Miyachi1990,Stein1993book,Chang1994,Chang2005} for full details.
Given a domain $\Omega\subset\mathbb{R}^d$, the Hardy space $\Hardy(\Omega)$ and space of bounded mean oscillation functions $\BMO(\Omega)$ are defined as follows.\footnote{In the notation of reference \cite{Chang2005}, we are defining $\Hardy_z(\Omega)$ and $\BMO_r(\Omega)$. We omit the subscripts to lighten the notation.} 
Starting with $\Hardy(\Omega)$, we fix a smooth and compactly supported function $\rho(\vec{x}) \geq 0$ with $\int_{\mathbb{R}^d} \rho\dVolume = 1$, and define the associated \emph{maximal function} operator by
\begin{equation}\label{e:max-function}
    M_\rho f(\vec{x}) = \sup_{\delta>0}\, \abs{\int_{\Omega}\frac{1}{\delta^{d}}\rho\left(\frac{\vec{x}-\vec{z}}{\delta}\right)f(\vec{z})\,d\vec{z} },\quad \vec{x}\in\mathbb{R}^d.
\end{equation}
This definition records the `worst-case averages' of a given function $f(\vec{x})$ against rescaled copies of the probability density $\rho$ (actually, it is the extension of $f$ by zero from $\Omega$ to $\mathbb{R}^d$ that is being averaged).
The \emph{Hardy space} $\smash{\Hardy(\Omega)}$ then consists of all $f\in L^1(\Omega)$ such that $M_\rho f\in L^1(\mathbb{R}^d)$, a condition that turns out to be independent of $\rho$. This is a Banach space under the norm
%
\begin{equation*}
    \|f\|_{\Hardy(\Omega)} = \int_{\mathbb{R}^d} M_\rho f(\vec{x})\dVolume
    \end{equation*}
which embeds continuously into $L^1(\Omega)$ per the inequality $\|\cdot\|_{L^1(\Omega)} \leq \|\cdot\|_{\Hardy(\Omega)}$ (a consequence of Lebesgue differentiation).
The reverse inequality fails, however, as an example based on approximating a Dirac mass shows. Indeed, let $\vec{x}_0\in \Omega$ and consider a sequence of functions $\{f_{\epsilon}\}$ that have $L^1$-norm equal to one, and are defined by taking $f_{\epsilon} = \epsilon^{-d}$ on the ball $B_{\epsilon}(\vec{x}_0)$ of radius $\epsilon>0$ centered at $\vec{x}_0$ and $f_\epsilon = 0$ otherwise. Taking $\delta(\vec{x}) \sim |\vec{x}-\vec{x}_0|\vee\epsilon$ in~\cref{e:max-function} yields the lower bound $M_\rho f_\epsilon(\vec{x}) \gtrsim \delta^{-d}(\vec{x})$, the $L^1$-norm of which diverges logarithmically as $\epsilon\to 0$. This calculation is at the heart of our pinching flow example in \cref{ss:pinching}.


Being a Banach space, $\Hardy(\Omega)$ has a dual. A famous result of Fefferman identifies $\Hardy(\Omega)^*$ with a function space introduced by John and Nirenberg \cite{John1961bmo} in connection with John's work on elasticity. The space of \emph{bounded mean oscillation} functions $\BMO(\Omega)$ consists of all functions $g(\vec{x})$ for which
%
\begin{equation*}
    \|g \|_{\BMO(\Omega)} = \sup_{Q\subset \Omega}\, \fint_Q \abs{g(\vec{x})-\fint_Q g} \dVolume <\infty
\end{equation*}
where $Q$ is a $d$-dimensional cube. Modulo constants, this is a norm under which $\BMO(\Omega)$ is a Banach space.
%
The duality between $\Hardy(\Omega)$ and $\BMO(\Omega)$ is realized by the inequality
\begin{equation}\label{eq:h1-bmo-duality-full-space}
    \abs{\int_\Omega f g \, \dVolume} \lesssim_{\Omega,d} \|f\|_{\Hardy(\Omega)} \|g\|_{\BMO(\Omega)}
\end{equation}
which holds at first for $f\in\Hardy(\Omega)$ and $g\in \BMO(\Omega)\cap L^\infty(\Omega)$, and then for all $g\in \BMO(\Omega)$ by continuous extension.
It follows directly from the definitions that $\|\cdot\|_{\BMO(\Omega)}\leq 2\|\cdot\|_{L^\infty(\Omega)}$ so that $L^\infty(\Omega)$ embeds continuously into $\BMO(\Omega)$. Again, the reverse direction fails: the function $\log(|\vec{x}-\vec{x}_0|)$ belongs to $\BMO(\Omega)$ (see \cite[Ch. IV]{Stein1993book}) but is not in $L^\infty(\Omega)$ if $\vec{x}_0\in\Omega$. This too shows up in our discussion of pinching flows. 

Finally, we recall the div-curl inequality of Coifman, Lions, Meyers and Semmes \cite{Coifman1990}: if $\vec{u}(\vec{x})$ and $\vec{v}(\vec{x})$ belong to $L^2(\mathbb{R}^d;\mathbb{R}^d)$ and are respectively divergence- and curl-free,    their inner product $\vec{u}\cdot\vec{v}$  belongs to $\Hardy(\mathbb{R}^d)$ and satisfies
\begin{equation}\label{eq:clms-full-space}
    \| \vec{u} \cdot \vec{v} \|_{\Hardy(\mathbb{R}^d)} \lesssim_d \|\vec{u}\|_{L^2(\mathbb{R}^d)}\|\vec{v}\|_{L^2(\mathbb{R}^d)}.
\end{equation}
The same result holds with a bounded Lipschitz domain $\Omega$ in place of $\mathbb{R}^d$ provided $\vec{u}$ satisfies no-penetration boundary conditions. The key points for deducing this from \cref{eq:clms-full-space} are that (i) with no-penetration conditions, the extension of $\vec{u}$ by $\vec{0}$ to $\mathbb{R}^d$ is  divergence-free, and (ii) one can find a curl-free extension of $\vec{v}$ to $\mathbb{R}^d$ whose $L^2$-norm is bounded by that of its restriction to $\Omega$ (apply the Sobolev extension theorem \cite{EvansPDE} to a potential $\varphi$ with $\vec{v}=\nabla \varphi$). We shall make repeated use of the resulting inequality, which states that
\begin{equation}\label{eq:clms-bdd-domain}
    \| \vec{u} \cdot \vec{v} \|_{\Hardy(\Omega)} \lesssim_{\Omega,d} \|\vec{u}\|_{L^2(\Omega)}\|\vec{v}\|_{L^2(\Omega)}
\end{equation}
if $\vec{u}$ is divergence-free with $\vec{u}\cdot\hat{\vec{n}}=0$ at $\partial\Omega$, and if $\vec{v}$ is curl-free.

\subsection{Bounding the heat transfer of energy-constrained flows}
\label{ss:lower-bounds-BMO-hardy}
\noindent
Combining the main result of \cref{ss:s2-unsteady-bounds} with the functional analysis recalled above, we bound  $\gradT$ from below in terms of the mean kinetic energy $\smallmean{|\vec{u}|^2}$. With an eye towards the examples of \cref{ss:lower-bounds-example}, we state this result for steady $f(\vec{x})$ while allowing $\vec{u}(\vec{x},t)$ and $T(\vec{x},t)$ to be unsteady (however, see the remark below).

\begin{corollary}\label{th:lb-BMO}
    Let $\vec{u}(\vec{x},t)$, $f(\vec{x})$ and $T(\vec{x},t)$ be as in  \cref{thm:unsteady-VP} (or as in \cref{rem:extra-integrability-velocity}).
    There is a constant $C>0$ depending only on $\Omega$ and $d$ such that
    \begin{equation}
        \label{e:lb:proof:general-xi}
        \left\langle |\nabla T|^{2}\right\rangle
        \geq
        \frac{ \left(\fint_\Omega {\xi}f \dVolume \right) ^{2}}{
        \fint_\Omega |\nabla{\xi}|^{2} \dVolume+
        C \|\xi\|_{\BMO(\Omega)}^2 \mean{| \vec{u} |^{2}}  }
    \end{equation}
for every non-constant $\xi\in H^1(\Omega)\cap L^\infty(\Omega)$ (or $H^1(\Omega)$, respectively).
\end{corollary}
\begin{remark}\label{rem:bmo-bound-unsteady-case}
The same bound holds for unsteady $f(\vec{x},t)$ with $\langle \xi f \rangle$ in place of $\fint_\Omega \xi f$, though if the time-average of $f$ vanishes identically then this is not a useful bound. To improve the result, one should use unsteady test functions $\xi(\vec{x},t)$ following \cref{thm:unsteady-VP}. This leads to a bound with an additional term $C \langle |\nabla \Delta^{-1}\partial_t\xi|^2\rangle$ in the denominator, the implications of which we leave to future work.
\end{remark}
\begin{proof}
Applying \cref{thm:unsteady-VP} with a steady test function $\xi(\vec{x})$ gives the lower bound
\[
\gradT \geq
         2\fint_\Omega f\xi\,\dVolume - \fint_\Omega \abs{\nabla\xi}^2\dVolume - \langle |\nabla\Delta^{-1} \vec{u}\cdot\nabla\xi|^2\rangle.
\]
Substituting $\lambda \xi$ for $\xi$ and optimizing $\lambda \in \mathbb{R}$, there follows
\begin{equation}\label{eq:step1-bound}
    \gradT \geq
    \frac{ \left ( \fint_\Omega f\xi \,\dVolume \right )^2}{\fint_\Omega \abs{\nabla\xi}^2\dVolume+\langle |\nabla\Delta^{-1} \vec{u}\cdot \nabla\xi )|^{2} \rangle }.
\end{equation}
Note the denominator is non-zero by our hypothesis on $\xi$.
We proceed to estimate $\smallmean{|\nabla\Delta^{-1}\vec{u}\cdot\nabla\xi) |^{2}}$. At almost every time,
    \begin{equation*}
        \int_\Omega |\nabla\Delta^{-1}\vec{u} \cdot \nabla \xi |^2
        \dVolume
        =  \max_{\varphi(\vec{x})}\,
            \left(
            \frac{ \int_{\Omega}\vec{u}\cdot\nabla\xi \,\varphi \, \dVolume}
            { \| \nabla \varphi \|_{L^2(\Omega)} }
            \right)^2
        =  \max_{\varphi(\vec{x})}\,
            \left(
            \frac{ \int_{\Omega}\vec{u}\cdot\nabla\varphi\,\xi \, \dVolume}
            { \| \nabla \varphi \|_{L^2(\Omega)} }
            \right)^2.
    \end{equation*}
By the duality of $\mathcal{H}^1$ and $\BMO$ in \cref{eq:h1-bmo-duality-full-space} and the div-curl inequality  \cref{eq:clms-bdd-domain},
    \begin{equation*}
        \abs{\int_{\Omega}\vec{u}\cdot\nabla\varphi \,\xi \,\dVolume}
        \lesssim_{\Omega,d}
        \|\vec{u} \cdot \nabla \varphi\|_{\Hardy(\Omega)} \|\xi\|_{\BMO(\Omega)}
        \lesssim_{\Omega,d}
        \|\vec{u}\|_{L^2(\Omega)}
        \|\nabla \varphi\|_{L^2(\Omega)}
        \|\xi\|_{\BMO(\Omega)}.
    \end{equation*}
    Combining these statements and averaging in time, there follows
   \[
        \langle |\nabla\Delta^{-1} \vec{u}\cdot \nabla\xi )|^{2} \rangle        \lesssim_{\Omega,d}
        \langle |\vec{u}|^2\rangle\cdot
        \|\xi\|_{\BMO(\Omega)}^2.
    \]
    Substituting into \cref{eq:step1-bound} yields the bound
    \[
        \left\langle |\nabla T|^{2}\right\rangle
        \geq
        \frac{ \left(\fint_\Omega {\xi}f \dVolume \right) ^{2}}{
        \fint_\Omega |\nabla{\xi}|^{2} \dVolume+
        C(\Omega,d)
        \|\xi\|_{\BMO(\Omega)}^2\langle |\vec{u}|^2\rangle }.
        \qedhere
    \]
\end{proof}

How should the test function $\xi(\vec{x})$ be chosen in this last result? The answer depends, of course, on the domain $\Omega$, the dimension $d$ and the structure of the source--sink function $f$. It also depends on the magnitude of $\langle |\vec{u}|^2\rangle$. On the one hand, for sufficiently small kinetic energies one expects to be able to `cross out' the second term in the denominator of \cref{e:lb:proof:general-xi}, and select $\xi$ through the maximization
 \begin{equation}\label{e:lb:xi-Hm1}
      \max_{\xi(\vec{x})}\, \frac{ \left(\int_\Omega {\xi}f \dVolume \right) ^{2} }
        {\int_\Omega |\nabla{\xi}|^{2} \dVolume  } = \|f\|_{H^{-1}(\Omega)}^2.
\end{equation}
This leads to the choice $\xi = \Delta^{-1}f$.
On the other hand, for large $\langle |\vec{u}|^2\rangle$ one is lead to the maximization
 \begin{equation}\label{e:lb:xi-Hardy}
      \max_{\xi(\vec{x})}\, \frac{ \left(\int_\Omega {\xi}f \dVolume \right) ^{2} }
        {\|\xi\|_{\BMO(\Omega)}^2 } \sim_{\Omega,d} \|f\|_{\Hardy(\Omega)}^2
\end{equation}
by the duality between $\Hardy(\Omega)$ and $\BMO(\Omega)$. Here the best choice of $\xi$ is less apparent, though one achieving this equivalence is always guaranteed to exist. 
(We guess that time-dependent $f$ could be handled similarly by a suitable smoothing in time of the choices in \cref{e:lb:xi-Hm1} and \cref{e:lb:xi-Hardy}, taking into account the additional term $\langle |\nabla \Delta^{-1}\partial_t \xi |^2 \rangle$ from \cref{rem:bmo-bound-unsteady-case}.)
Of course, once one makes a choice for $\xi$, it can be plugged back into \cref{e:lb:proof:general-xi} to achieve a lower bound with known constants at all values of $\langle |\vec{u}|^2\rangle$. We demonstrate this in examples below.


\subsection{Two examples}
\label{ss:lower-bounds-example}
\noindent We now apply our variational bounds to a pair of examples involving oscillatory or concentrated heating and cooling. In each example, we deduce the scaling law of $\min\, \gradT$ with respect to its parameters, along with velocity fields achieving the optimal scalings.  See \cref{ss:hf-rolls} for oscillatory heating and our accompanying cellular flows, and \cref{ss:pinching} for concentrated heating and our pinching flows.
\subsubsection{Sinusoidal heating and cellular flows}
\label{ss:hf-rolls}
\noindent
Our first example optimizes heat transfer between a periodic pattern of sources and sinks. Let $\Omega = (0,2\pi)^2$ and take
\begin{equation}\label{e:ex:rolls:force}
    f (\vec{x}) =
     \frac12\cos\left( \frac{2y}{\ell} \right)
    -\frac12\cos\left( \frac{2x}{\ell} \right).
\end{equation}
The parameter $\ell^{-1}\in\mathbb{N}$ sets the period of the pattern.

\begin{proposition}\label{th:cellular-flows}
    Under the above setup,
    \begin{equation*}
        \min_{\substack{\vec{u}(\vec{x},t) \\ \smallmean{|\vec{u}|^2} \leq \Pe^2 \\ \partial_t T + \vec{u}\cdot\nabla T = \Delta T + f}} \gradT \sim \min\left\{\ell^2,\frac{1}{\Pe^2} \right\}
    \end{equation*}
for all $\ell^{-1}\in\mathbb{N}$ and $\Pe\geq 0$. The alternatives are achieved by no flow ($\ell^2$) or by the cellular flow ($\Pe^{-2}$) depicted in \cref{fig:flows}a of the introduction.
\end{proposition}

\begin{proof}[Proof of the lower bound]

We begin with the general lower bound
\begin{equation}\label{e:ex:rolls:lb-inf-norm}
    \left\langle |\nabla T|^{2}\right\rangle
    \geq
    \frac{ \left(\fint_\Omega {\xi}f \dVolume \right) ^{2}}{
    \fint_\Omega |\nabla{\xi}|^{2} \dVolume+
    C(\Omega) \|\xi\|_{\BMO(\Omega)}^2  \Pe^2  }
\end{equation}
from \cref{th:lb-BMO}.
The present $f$ belongs to $L^\infty(\Omega)$ and is such that all of its $L^p$-norms are comparable. In particular, $\|f\|_{L^1(\Omega)} \sim \|f\|_{L^\infty(\Omega)}\sim 1$ for all $\ell$. Also, $\|f\|_{\Hardy(\Omega)} \sim \|f\|_{\BMO(\Omega)}\sim 1$ and so there exist many good choices of $\xi$.

Take, for example, $\xi = f$. Then
\[
\int_\Omega \xi f\dVolume = \int_\Omega f^2 \dVolume = \int_\Omega \abs{\frac12\cos\left( \frac{2y}{\ell} \right)
    -\frac12\cos\left( \frac{2x}{\ell} \right)}^2\dVolume \sim 1
\]
while
\[
\|f\|_{\BMO(\Omega)}\leq 2 \|f\|_{L^\infty(\Omega)}\leq 2.
\]
Also,
\[
\int_\Omega |\nabla \xi|^2 = \int_\Omega |\nabla f|^2 = \int_\Omega \abs{-\frac{1}{\ell}\sin\left( \frac{2y}{\ell} \right)\hat{\vec{e}}_y
    -\frac{1}{\ell}\sin\left( \frac{2x}{\ell} \right)\hat{\vec{e}}_x}^2\dVolume \sim \frac{1}{\ell^2}.
\]
Combining these estimates into \cref{e:ex:rolls:lb-inf-norm} yields the lower bound
    \begin{equation*}
        \left\langle |\nabla T|^{2}\right\rangle
        \gtrsim
        \frac{1}{
        \ell^{-2}+
         \Pe^2  }
         \gtrsim \min\left\{ \ell^2, \frac{1}{\Pe^2} \right\}.
    \end{equation*}

\paragraph{Proof of the upper bound} We seek a steady velocity $\vec{u}(\vec{x})$ whose thermal dissipation is similar to the lower bound. To guide the search, consider the upper bound
 \[
    \left\langle \abs{\nabla T}^2\right\rangle
        \leq
            \fint_\Omega \abs{\nabla\eta}^2\dVolume
            +\fint_\Omega |\nabla\Delta^{-1}(\vec{u}\cdot\nabla\eta-f)|^{2}
            \dVolume    %
\]
from \cref{th:steady-VP}, which holds in the present two-dimensional case for all $\eta\in H^1(\Omega)$.
Making the change of variables

\begin{equation*}
    \vec{u}=\frac{\Pe}{\sqrt{\fint_{\Omega}|\tilde{\vec{u}}|^{2}\dVolume}}\  \tilde{\vec{u}}\quad\text{and}\quad\eta=\frac{\sqrt{\fint_{\Omega}|\tilde{\vec{u}}|^{2}\dVolume}}{\Pe}\ \tilde{\eta}
\end{equation*}
and dropping the tildes yields the estimate
\begin{equation}
    \left\langle |\nabla T|^{2}\right\rangle \leq\frac{1}{\Pe^{2}}
    \fint_\Omega |\vec{u}|^{2} \dVolume
     \fint_\Omega |\nabla \eta|^{2} \dVolume
    +\fint_\Omega \abs{\nabla\Delta^{-1}(\vec{u}\cdot\nabla\eta-f)}^{2}\dVolume
    \label{eq:upper-bound-Pe}
\end{equation}
for all $\vec{u}$ and $\eta$. This reformulation simplifies the algebra, as it allows us to neglect the kinetic energy constraint. Of course, it is actually the unscaled velocity with kinetic energy equal to $\Pe$ whose thermal dissipation we are estimating.

There are two alternatives to consider, depending on whether we should take $\vec{u}=\vec{0}$ or not. In the case with no flow, the choice of $\eta$ is immaterial and
\begin{equation*}
        \gradT
        \leq \fint_\Omega \abs{\nabla \Delta^{-1} f}^2 \dVolume
        = \fint_\Omega \abs{
        -\frac{\ell}{4}\sin\left(\frac{2x}{\ell}\right)\hat{\vec{e}}_{x}
        +\frac{\ell}{4}\sin\left(\frac{2y}{\ell}\right)\hat{\vec{e}}_{y}
        }^2 \dVolume
        \sim \ell^2
    \end{equation*}
with $\hat{\vec{e}}_x$ and $\hat{\vec{e}}_y$ being the unit vectors along the $x$- and $y$-coordinates.
On the other hand, for the particular $f$ in the example we can easily construct an admissible pair $(\vec{u},\eta)$ satisfying the pure advection equation
\[
\vec{u}\cdot\nabla \eta = f.
\]
Simply take $\vec{u}=\nabla^\perp \psi = (\partial_y \psi, -\partial_x\psi)$ with the stream function
    \begin{equation*}
       \psi(\vec{x}) =l\sin\left(\frac{x}{l}\right)\sin\left(\frac{y}{l}\right)
    \end{equation*}
    and use the test function
    \begin{equation*}
        \eta(\vec{x})  =-l\cos\left(\frac{x}{l}\right)\cos\left(\frac{y}{l}\right).
    \end{equation*}
In fact, the definition of $f$ in \cref{e:ex:rolls:force} was made precisely with these choices in mind. The second term in \cref{eq:upper-bound-Pe} now vanishes, so that
\[
\left\langle |\nabla T|^{2}\right\rangle \leq
\frac{1}{\Pe^{2}}    \fint_\Omega |\vec{u}|^{2} \dVolume      \fint_\Omega |\nabla \eta|^{2} \dVolume
\lesssim \frac{1}{\Pe^2}.
\]
Since we are always free to use either velocity field, the minimum thermal dissipation is bounded according as
\[
\min\,\left\langle |\nabla T|^{2}\right\rangle \lesssim \min\left\{ \ell^2, \frac{1}{\Pe^2} \right\}.
\]
The proof is complete.
\end{proof}

\subsubsection{Concentrated heating and pinching flows}
\label{ss:pinching}
\noindent
Next we consider source--sink profiles of the general form
\begin{equation*}
    f(\vec{x}) = f_+(\vec{x}) - f_-(\vec{x})
\end{equation*}
where $f_\pm$ are non-negative and supported in disjoint balls $B_\epsilon(\vec{x}_\pm)$ centered at $\vec{x}_\pm$ with radii $\epsilon>0$. Fixing units, we take
\begin{equation*}
    \int_{B_\epsilon(\vec{x}_\pm)} f_\pm(\vec{x}) \dVolume =
    1
\end{equation*}
and $\vec{x}_\pm = (0,\pm 1/2)$. We also suppose that
\begin{equation*}
    \|f_\pm\|_{L^\infty(B_\epsilon(\vec{x}_\pm))} \lesssim \frac{1}{\epsilon^2}
    \qquad \text{and} \qquad
    \|\nabla f_\pm\|_{L^\infty(B_\epsilon(\vec{x}_\pm))} \lesssim \frac{1}{\epsilon^3}
\end{equation*}
and impose the `up-down' symmetry condition
\begin{equation}\label{e:ex:pinching:symmetry}
    f_+(x,y) = f_-(x,-y)
\end{equation}
saying that the heat added by $f_+$ at $(x,y)$ matches the heat taken away by $f_-$ at $(x,-y)$.
A source--sink distribution satisfying these conditions can be constructed by smoothing a point source and point sink across a scale $\sim\epsilon$; there are of course many other possibilities.  Regarding the domain, we assume for simplicity that it is the square $\Omega=(-1,1)^2$.

\begin{proposition}\label{th:pinching}
   Under the above setup,
   \begin{equation*}
        \min_{\substack{\vec{u}(\vec{x},t) \\ \smallmean{|\vec{u}|^2} \leq \Pe^2 \\ \partial_t T + \vec{u}\cdot\nabla T = \Delta T + f}} \gradT \sim \min
        \left\{
            \log\frac1\epsilon,\,
            \left(\log\frac1\epsilon\right)^2\frac{1}{\Pe^2}
            \right\}
    \end{equation*}
   for all $\epsilon\in (0,1/20)$ and $\Pe \geq 0$.
   The alternatives are achieved by no flow ($\log(\epsilon^{-1})$) or by the pinching flow ($\log^2(\epsilon^{-1})\Pe^{-2}$) depicted in \cref{fig:flows}b of the introduction.
\end{proposition}

\begin{remark}
Our setup is already quite general, but one can relax it further without altering the scaling of the result. This includes allowing the symmetry condition \cref{e:ex:pinching:symmetry} to hold only after integration in $x$, or considering general domains $\Omega$ that include the pinching flows we use to prove the upper bound.
\end{remark}

\begin{proof}[Proof of the lower bound] Again we begin with the lower bound
 \begin{equation}\label{e:ex:pinch:lb-inf-norm}
        \left\langle |\nabla T|^{2}\right\rangle
        \geq
        \frac{ \left(\fint_\Omega {\xi}f \dVolume \right) ^{2}}{
        \fint_\Omega |\nabla{\xi}|^{2} \dVolume+
        C(\Omega) \|\xi\|_{\BMO(\Omega)}^2  \Pe^2  }
    \end{equation}
from \cref{th:lb-BMO}.
Recall the example of the smoothed Dirac mass discussed in \cref{ss:BMO-bound}, which had logarithmically diverging $\mathcal{H}^1$-norm as $\epsilon \to 0$. This prompts us to look for a test function $\xi(\vec{x})$ with the properties that
    \begin{equation*}
        \int_\Omega \xi f \dVolume \gtrsim \log \frac1\epsilon \quad\text{and}\quad \|\xi\|_{\BMO(\Omega)}\lesssim 1,
    \end{equation*}
which would prove in the present setting that $\|f\|_{\mathcal{H}^1} \gtrsim \log \epsilon^{-1}$.    %
    A suitable choice is given by
    \begin{equation*}
    \xi(\vec{x}) = \begin{cases}
        \xi_0(|\vec{x}-\vec{x}_+|)
        &\text{if } |\vec{x}-\vec{x}_+|\leq \frac14\\
        -\xi_0(|\vec{x}-\vec{x}_-|)
        &\text{if } |\vec{x}-\vec{x}_-|\leq \frac14\\
        0
        &\text{otherwise}
    \end{cases}
    \qquad\text{where}\qquad
    \xi_0(r) = \begin{cases}
        \log(\frac{1}{4\epsilon}) &\text{if }r\leq\epsilon\\
        \log(\frac{1}{4r}) &\text{if }\epsilon < r \leq 1/4\\
        0 &\text{otherwise}
    \end{cases}.
    \end{equation*}
For one,
%
\begin{equation*}
\int_{\Omega}\xi f\dVolume =\int_{B_{\epsilon}(\vec{{x}}_{+})}\xi_{0}(|\vec{x}-\vec{x}_{+}|)f_{+}(\vec{x})\dVolume+\int_{B_{\epsilon}(\vec{{x}}_{-})}\xi_{0}(|\vec{x}-\vec{x}_{-}|)f_{-}(\vec{x})\dVolume
 =2\log\left(\frac{1}{4\epsilon}\right)\gtrsim \log\frac{1}{\epsilon}.
\end{equation*}
    %
Also, $\|\xi\|_{\BMO(\Omega)}\lesssim 1$ as $\log(|\vec{x}|)\in \BMO(\mathbb{R}^d)$, and since the minimum and maximum of two functions $g,h\in \BMO(\mathbb{R}^d)$ have $\BMO$-norms bounded by a multiple of $\|g\|_{\BMO(\mathbb{R}^d)}+\|h\|_{\BMO(\mathbb{R}^d)}$ (see \cite[Ch. IV]{Stein1993book}).

Continuing, we compute the $H^1$-norm in the dominator of \cref{e:ex:pinch:lb-inf-norm}. Evidently,
\begin{equation*}
    \int_{\Omega}|\nabla\xi|^{2}\dVolume =\int_{\Omega\backslash B_{\epsilon}(\vec{{x}}_{+})}|\nabla\xi_{0}(|\vec{x}-\vec{x}_{+}|)|^{2}\dVolume+\int_{\Omega\backslash B_{\epsilon}(\vec{{x}}_{-})}|\nabla\xi_{0}(|\vec{x}-\vec{x}_{-}|)|^{2}\dVolume
    \lesssim\int_{r=\epsilon}^{r=1/4}\frac{1}{r}\,dr
    \lesssim\log\frac{1}{\epsilon}.
\end{equation*}
Assembling the estimates shows that
    \begin{equation*}
        \gradT
        \gtrsim \frac{\left(\log\frac1\epsilon\right)^2}{\log\frac1\epsilon + \Pe^2}
        \sim \min\left\{
            \log\frac1\epsilon,\,
            \left(\log\frac1\epsilon\right)^2\frac{1}{\Pe^2}
        \right\}.
    \end{equation*}

    \paragraph{Proof of the upper bound}
We turn to construct steady velocity fields $\vec{u}(\vec{x})$  saturating the lower bound. Arguing just as in proof of \cref{th:cellular-flows} (see the derivation leading up to \cref{eq:upper-bound-Pe}) we apply \cref{th:steady-VP} to show that
 \begin{equation}
    \left\langle |\nabla T|^{2}\right\rangle \leq\frac{1}{\Pe^{2}}
    \fint_\Omega |\vec{u}|^{2} \dVolume
     \fint_\Omega |\nabla \eta|^{2} \dVolume
    +\fint_\Omega \abs{\nabla\Delta^{-1}(\vec{u}\cdot\nabla\eta-f)}^{2}\dVolume
    \label{eq:upper-bound-Pe-secondex}
\end{equation}
where $T$ is the temperature field associated to the scaled version of $\vec{u}$ with mean kinetic energy $\Pe$. Again, this upper bound applies for any choice of $\vec{u}$ and $\eta$, regardless of the $L^2$-norm of the velocity. We shall consider two different choices for $(\vec{u},\eta)$, the first of which involves no flow, and the second of which is the anticipated pinching flow.

\paragraph{No flow}
The first possibility is to take $\vec{u}=\vec{0}$. Then $\eta$ drops out in \cref{eq:upper-bound-Pe-secondex}, and we see that
    \begin{equation*}
        \gradT \leq \fint_\Omega \abs{\nabla \Delta^{-1}f}^2 \dVolume = \max_{\substack{\varphi(\vec{x})\\
\fint_{\Omega}\varphi\dVolume=0
}
}\,\frac{\left|\int_{\Omega}\varphi_{\pm}f_{\pm}\right|^{2}}{\fint_{\Omega}|\nabla\varphi|^{2}}.
    \end{equation*}
To prove that $\gradT \lesssim \log(\epsilon^{-1})$, which is the desired upper bound in this case, it suffices to show that
\begin{equation}\label{e:ex:pinching:no-flow:key-estimate}
\left|\int_{\Omega}\varphi_{\pm}f_{\pm}\dVolume\right|^2\lesssim \log\left(\frac{1}{\epsilon}\right) \int_\Omega |\nabla \varphi|^2\dVolume.
\end{equation}
Here, $\varphi_\pm$ are the positive and negative parts of $\varphi$, and we allow for any combination of pluses and minuses on the left (e.g., $\varphi_+f_-$).
Since the argument is the same for all combinations, we use $\varphi_+f_+$.
By our assumptions on $f$,
    \begin{equation*}
       \int_\Omega \varphi_+f_+\dVolume =  \int_{B_\epsilon(\vec{x}_+)} \varphi_+ f_+ \dVolume
        \lesssim
        \fint_{Q_\epsilon(\vec{x}_+)} \varphi_+ \dVolume
            \end{equation*}
where $Q_\epsilon(\vec{x}_+)$ is the open square of side length $\epsilon$ centered at $\vec{x}_+$. The desired bound now follows from a $\BMO$-type argument, involving controlling consecutive jumps in the average of $\varphi_+$ along a sequence of squares starting with $Q_{\epsilon}(\vec{x}_+)$ and ending at $\Omega$.

Since $\Omega=(-1,1)^2$, there is a sequence of squares $Q_1,\dots,Q_N\subset\Omega$ of ever increasing diameters and with the following properties: (i) the first square is $Q_1=Q_{\epsilon}(\vec{x}_+)$ and the last square is $\Omega$; (ii) consecutive squares intersect, with an area $|Q_i\cap Q_{i+1}|$ that is within a factor of $5$ of the areas $|Q_i|$ and $|Q_{i+1}|$;
(iii) no more than $10$ squares include any given $\vec{x}\in\Omega$; 
(iv) there are $N\sim \log(1/\epsilon)$ squares in total. 
To use the squares, observe first that
\begin{align*}
\left|\fint_{Q_{i}}\varphi_{+}\dVolume-\fint_{Q_{i+1}}\varphi_{+}\dVolume\right|^{2} & \lesssim\fint_{Q_{i}}\left|\varphi_+-\fint_{Q_{i}}\varphi_{+}\right|^{2}\dVolume+\fint_{Q_{i+1}}\left|\varphi_+-\fint_{Q_{i+1}}\varphi_{+}\right|^{2}\dVolume\\
 & \lesssim\fint_{Q_{i}\cup Q_{i+1}}|\nabla\varphi_{+}|^{2}\dVolume
\end{align*}
by condition (ii) and Poincar\'e's inequality for $d=2$. Summing up over consecutive pairs of squares,
\begin{align*}
\left|\fint_{Q_{\epsilon}(\vec{x}_{+})}\varphi_{+}\dVolume-\fint_{\Omega}\varphi_{+}\dVolume\right|^{2} & =\left|\sum_{i=1}^{N-1}\fint_{Q_{i}}\varphi_{+}\dVolume-\fint_{Q_{i+1}}\varphi_{+}\dVolume\right|^{2}
  \lesssim N\sum_{i=1}^{N-1}\left|\fint_{Q_{i}}\varphi_{+}\dVolume-\fint_{Q_{i+1}}\varphi_{+}\dVolume\right|^{2}\\
 & \lesssim N\sum_{i=1}^{N-1}\int_{Q_{i}\cup Q_{i+1}}|\nabla\varphi_{+}|^{2} \lesssim N\int_{\Omega}|\nabla\varphi_{+}|^{2}
\end{align*}
where in the first step we used condition (i), in the second step we applied the Cauchy--Schwarz inequality and in the last step we used condition (iii). Since
\[
\abs{\fint_\Omega \varphi_+\dVolume}^2\lesssim  \int_\Omega |\varphi|^2\dVolume  \lesssim \int_\Omega |\nabla \varphi|^2\dVolume
\]
we can conclude the result. In particular,
\begin{align*}
\abs{\int_\Omega \varphi_+f_+}^2 &\lesssim \abs{\fint_{Q_\epsilon(\vec{x}_+)}\varphi_+}^2 \lesssim  \abs{\fint_{\Omega}\varphi_+}^2 + \abs{\fint_{Q_\epsilon(\vec{x}_+)}\varphi_+ - \fint_{\Omega}\varphi_+}^2 \\
&\lesssim (1+N)\int_\Omega |\nabla \varphi|^2 \lesssim \log\left(\frac{1}{\epsilon}\right)\int_\Omega |\nabla \varphi|^2
\end{align*}
by condition (iv). This shows \cref{e:ex:pinching:no-flow:key-estimate} and hence
    \begin{equation*}
        \gradT \lesssim \log\frac1\epsilon
    \end{equation*}
for the choice $\vec{u}=\vec{0}$.

\paragraph{Pinching flows}
Next we show how to achieve $ \gradT \lesssim \log^2(\epsilon^{-1})\Pe^{-2}$ using a `pinching' flow. The flow we have in mind squeezes a large portion of the domain $\Omega$ into the balls $B_\epsilon(\vec{x}_{\pm})$ where the heat is being added and taken away. This requires the velocity to grow as $1/|\vec{x}-\vec{x}_\pm|$, which results in a logarithmically diverging kinetic energy. At the same time we will enforce the pure advection equation $\vec{u}\cdot\nabla \eta = f$ leading to a similar divergence in the homogeneous $H^1$-norm of $\eta$. Using all of this in the bound
\begin{equation}\label{e:ex:pinching:ub-2}
\left\langle |\nabla T|^{2}\right\rangle \leq\frac{1}{\Pe^{2}}
\fint_\Omega |\vec{u}|^{2} \dVolume
 \fint_\Omega |\nabla \eta|^{2} \dVolume
\end{equation}
which follows from \cref{eq:upper-bound-Pe-secondex} will lead to the desired result.

The key task is to find a way to solve the pure advection equation with the given source--sink functions $f=f_+-f_-$. Our solution will be symmetric under the reflection $(x,y) \mapsto (x,-y)$, so we define it explicitly on the upper-half plane.
Introduce polar coordinates $(r,\theta)$ centered at $(0,1/2+2\epsilon)$, and let $R_\epsilon$ be the rectangle centered at $\vec{x}_+=(0,1/2)$ with vertical side length $2\epsilon$ and horizontal side length $2\smash{\sqrt{3}}\epsilon$.
The rectangle is defined such that it contains the ball $B_\epsilon(\vec{x}_+)$ where the source $f_+$ is supported, and is such that its top and bottom sides are tangent to this ball. Outside of $R_\epsilon$ and for $y>0$, we define $\vec{u}=\nabla^\perp\psi_1$ with
    \begin{equation*}
        \psi_1(\theta) = \begin{cases}
            \frac{11\pi}{6}-\theta & \theta\in(\frac{5\pi}{3},\frac{11\pi}{6}]\\
            \frac{\pi}{6} & \theta\in(\frac{4\pi}{3},\frac{5\pi}{3}]\\
            \theta-\frac{7\pi}{6} & \theta\in(\frac{7\pi}{6},\frac{4\pi}{3}]\\
            0 & \text{otherwise}
        \end{cases}.
    \end{equation*}
    The streamlines are left-right symmetric and are arranged in two trapezoidal channels, and the flow enters the rectangle $R_\epsilon$ from the right and exits it on the left. Inside $R_\epsilon$, we use a horizontal flow that matches the inflow and outflow conditions of the prior construction on the vertical sides of $R_\epsilon$. Specifically, we take $\vec{u}=\nabla^\perp\psi_2$ with
    \begin{equation*}
        \psi_2(y) =
        \psi_1\left(2\pi - \arctan\left(\frac{2y-1-4\epsilon}{2\sqrt{3}\epsilon}\right)\right)
        = -\frac{\pi}{6} - \arctan\left(\frac{2y-1-4\epsilon}{2\sqrt{3}\epsilon}\right).
    \end{equation*}
The rest of the flow is defined by odd reflection across the line $y=0$.

Having chosen $\vec{u}$, we now show how to solve $\vec{u} \cdot \nabla \eta = f$ to produce the required test function $\eta$. Inside $R_\epsilon$  the equation simplifies to $\partial_y \psi_2 \partial_x \eta = f$, which we integrate to get $\eta=\eta_2$ with 
\begin{equation*}
    \eta_2(\vec{x}) = \int_0^x \frac{f(s,y)}{\partial_y \psi_2(y)} \,{\rm d}s
    = -\int_0^x  \frac{12\epsilon^2 + (2y-1-4\epsilon)^2}{4\sqrt{3}\epsilon}\,f(s,y)\,{\rm d}s.
\end{equation*}
%
Outside of $R_\epsilon$ and for $y>0$, we take $\eta=\eta_1=0$ in regions of no flow, and choose $\eta_1$ to be otherwise constant along the streamlines. Matching conditions are imposed to ensure continuity across the boundary of $R_\epsilon$. In formulas,
\begin{equation*}
        \eta_1(\theta) = \begin{cases}
                \eta_2(\sqrt{3}\epsilon,\frac12 + 2\epsilon + \sqrt{3}\epsilon\tan\theta)
                &\text{if }\theta \in (\frac{5\pi}{3},\frac{11\pi}{6}]
                \\
                \eta_2(-\sqrt{3}\epsilon,\frac12 + 2\epsilon - \sqrt{3}\epsilon\tan\theta)
                &\text{if }\theta \in (\frac{7\pi}{6},\frac{4\pi}{3}]
                \\
                0
                &\text{otherwise}
        \end{cases}.
\end{equation*}
This gives $\vec{u} \cdot \nabla \eta_1 = 0 $ outside of $R_\epsilon$, and then we define $\eta$ for $y<0$ by reflection about $y=0$. Altogether, we have produced a pair $(\vec{u},\eta)$ solving the pure advection equation $\vec{u}\cdot\nabla\eta = f$ on $\Omega$.

To complete the proof we must estimate the $L^2$-norms of $\vec{u}$ and $\nabla \eta$.  By the up-down symmetry,
    \begin{align*}
        \int_\Omega \abs{\vec{u}}^2 \dVolume
        &\lesssim
        \int_{(\Omega \setminus R_\epsilon) \cap\{y> 0\} } \abs{\nabla \psi_1}^2 \dVolume +  \int_{R_\epsilon} \abs{\nabla \psi_2}^2 \dVolume
        \\
        &\lesssim
        \int_0^{2\pi}\int_{2\sqrt{3}}^\frac12 \abs{\partial_r \psi_1}^2 + \abs{\frac{1}{r}\partial_\theta \psi_1}^2 \, r {\rm d}r{\rm d}\theta
        +
        \int_{R_\epsilon} \abs{\frac{4\sqrt{3}\epsilon}{12\epsilon^2 + (2y-1-4\epsilon)^2}}^2 \dVolume
        \\
        &\lesssim
        \log\left(\frac1\epsilon\right) + 1
    \end{align*}
A similar calculation using the bounds $\abs{f} \lesssim \epsilon^{-2}$ and $\abs{\nabla f} \lesssim \epsilon^{-3}$ assumed at the start of the example give that \begin{equation*}
        \int_\Omega \abs{\nabla \eta}^2 \dVolume\lesssim        \log\left(\frac1\epsilon\right) + 1.
\end{equation*}
Plugging these estimates into \cref{e:ex:pinching:ub-2} shows that
\[
\left\langle |\nabla T|^{2}\right\rangle \lesssim \frac{\log^2(1/\epsilon)}{\Pe^2}
\]
for our pinching flow.

Using the better of the two flows --- no flow or the pinching flow --- bounds the minimum thermal dissipation  by
\[
 \min\,\gradT \lesssim \min
        \left\{
            \log\frac1\epsilon,\,
            \left(\log\frac1\epsilon\right)^2\frac{1}{\Pe^2}
            \right\}.
\]
The proof is complete. \end{proof}

\section{Asymptotic analysis of steady optimal flows}
\label{sec:asymptotics}
%

\noindent Each of the lower bounds from the previous section rearranges to give an asymptotic result: given a sequence $\{(\vec{u}_n,T_n)\}$ solving the advection-diffusion equation with source--sink $f(\vec{x})$ and with $\langle |\vec{u}_n|^2\rangle \to \infty$,
\begin{equation*}
    \liminf_{n\to\infty}\, \langle |\vec{u}_n|^2\rangle \cdot \langle |\nabla T_n|^2 \rangle \gtrsim_{\Omega,d} \|f\|_{\Hardy(\Omega)}^2 > 0.
\end{equation*}
The cellular and pinching flow examples from \cref{ss:lower-bounds-example} give scenarios in which this bound is sharp in its scaling with respect to the mean kinetic energy $\langle |\vec{u}_n|^2\rangle$, as well as features of $f$. Motivated by this, we now ask what it takes for a sequence of velocity fields to be `almost optimal' in the sense that their thermal dissipation is minimized at leading order. Focusing on the fully steady case where $\vec{u}=\vec{u}(\vec{x})$,  $f=f(\vec{x})$ and $T=T(\vec{x})$, we obtain a limiting variational problem whose minimizers encode key asymptotic properties of almost minimizers (including minimizers as a special case). The minimum value of this problem gives the sharpest possible asymptotic lower bound.

A word about setup is required, especially regarding the regularity of our velocity fields. Depending on the application, one may wish to constrain a different norm of the velocity other than the kinetic energy-based $L^2$-one we have used so far (e.g., the convection problem treated in \cref{sec:buoyancy} lends itself to the $H^1$-norm). In this section, we consider divergence-free and no-penetration velocities $\vec{u}$ belonging to a general Banach space $(X,\|\cdot\|_X)$, which for a technical reason we must assume is continuously embedded into $L^d(\Omega)$ via the inequality $\|\cdot\|_{L^d(\Omega)}  \lesssim \|\cdot \|_X$. 
We further assume $X$ is a dual space, so that its unit ball $\|\cdot\|_X\leq 1$ is weak-$*$ compact  \cite{Brezis2011}; 
this ensures the existence of optimizers for the problems we consider below.
Following \cref{rem:reduced-regularity}, we let $f\in H^{-1}(\Omega)$ and be mean-free. 

Given this setup, we ask to take the parameter $\Pe \to \infty$ in the sequence of minimization problems
\begin{equation}\label{eq:the-minimization-problem}
\min_{\substack{\vec{u}(\vec{x}) \\ \|\vec{u}\|_X\leq\Pe \\ \vec{u}\cdot\nabla T = \Delta T + f}}
\, \fint_\Omega |\nabla T|^2\dVolume.
 \end{equation}
Applying the sharp variational upper bound from \cref{th:steady-VP}, we learn that
  \begin{equation}\label{eq:sharp-integral-formulation}
        \min_{\substack{\vec{u}(\vec{x}) \\ \|\vec{u}\|_X\leq\Pe \\ \vec{u}\cdot\nabla T = \Delta T + f}}
\, \fint_\Omega |\nabla T|^2\dVolume
        =
        \min_{\substack{
            \vec{u}(\vec{x}),\eta(\vec{x})\\
            \|\vec{u}\|_{X}\leq\Pe
        }
        }\, \fint_\Omega |\nabla\eta|^{2} +  | \nabla \Delta^{-1}(\vec{u} \cdot \nabla \eta-f)|^{2} \dVolume
    \end{equation}
where the admissible $\eta$ belong to $H^1(\Omega)$. The differential equation on the left-hand side is embedded in the optimization on the right. It follows from the right-hand formulation that optimal velocities achieve $\|\vec{u}\|_X = \Pe$ if $f$ is not identically zero, since otherwise one could decrease the minimum by replacing $(\vec{u},\eta)$ with $(\lambda \vec{u},\lambda^{-1}\eta)$ for some $\lambda >1$.

First, we identify a sufficient and necessary condition for the minimum to scale as $\Pe^{-2}$.
\begin{lemma}\label{lem:bdd-limit} There holds
        \[
    \limsup_{\Pe \to \infty} \min_{\substack{\vec{u}(\vec{x}) \\ \|\vec{u}\|_X\leq\Pe \\ \vec{u}\cdot\nabla T = \Delta T + f}}
\, \Pe^2 \fint_\Omega |\nabla T|^2\dVolume
     <\infty
    \]
  if and only if there exists $(\vec{u}_0,T_0)\in X\times H^1(\Omega)$ satisfying
    \begin{equation}
    \label{eq:examples:pure-advection}
        \begin{cases}
        \vec{u}_{0}\cdot\nabla T_{0}=f & \text{in }\Omega\\
        \nabla\cdot\vec{u}_0=0 & \text{in }\Omega\\
        \vec{u}_{0}\cdot\hat{{\vec{{n}}}}=0 & \text{at }\partial\Omega
        \end{cases}.
    \end{equation}
\end{lemma}

\begin{proof}
   That the existence of $(\vec{u}_0,T_0)$ implies the asserted $\Pe^{-2}$ bound follows from the right-hand formulation of the optimization in \cref{eq:sharp-integral-formulation}. Indeed, we can always assume that $\|\vec{u}_0\|_X = 1$, and then setting $(\vec{u},\eta)= (\Pe \vec{u}_0,\Pe^{-1}T_0)$ into \cref{eq:sharp-integral-formulation} shows that $\min\, \fint_\Omega |\nabla T|^2 \leq \Pe^{-2}\fint_\Omega |\nabla T_0|^2$ for all $\Pe$.

    For the reverse implication, let $\{\vec{u}_{\Pe}\}$ be an admissible sequence for the finite-$\Pe$ problems, with $\|\vec{u}_{\Pe}\|_X \leq \Pe$ and whose temperatures $\{T_{\Pe}\}$ obey $ \fint_\Omega |\nabla T_{\Pe}|^2 \lesssim \Pe^{-2}$. Rescale to the variables $(\tilde{\vec{u}}_{\Pe}, \tilde{T}_{\Pe}) := (Pe^{-1}\vec{u}_{\Pe},\Pe T_{\Pe})$ to find that
    \[
          \|\vec{u}_{\Pe}\|_X = 1,\quad \|\nabla T_{\Pe}\|_{L^2(\Omega)} \leq 1\quad\text{and}\quad  \vec{u}_{\Pe} \cdot \nabla T_{\Pe} = \frac{1}{\Pe} \Delta T_{\Pe} + f
    \]
after dropping the tildes. Applying the Banach-Alaoglu theorem \cite{LaxFunctionalAnalysis} to the dual Banach space $X$ and using our assumption that it is continuously embedded into $L^d(\Omega)$, hence also in $L^2(\Omega)$ since $d\geq 2$, we can extract a subsequence $\{\vec{u}_{\Pe}, T_{\Pe}\}$ (not relabeled)  converging weakly-$*$ to $(\vec{u}_0, T_0)$ both in $X\times H^1(\Omega)$ and in $L^2(\Omega)\times H^1(\Omega)$. Note
\begin{equation*}
\|\vec{u}_{\Pe}\cdot\nabla T_{\Pe}-f\|_{H^{-1}(\Omega)} =\left\|\frac{1}{\Pe}\Delta T_{\Pe} \right\|_{H^{-1}(\Omega)}
=\frac{1}{\Pe}\|\nabla T_{\Pe}\|_{L^{2}(\Omega)}
\leq\frac{1}{\Pe}
\to0
\end{equation*}
by the definition of the $H^{-1}$-norm in \cref{e:Hm1-norm}. An application of the div-curl lemma~\cite[Theorem 4 in \S5.B]{Evans1990} then verifies that the dot product $\vec{u}_{\Pe} \cdot \nabla T_{\Pe}$ converges to $\vec{u}_0 \cdot \nabla T_0$, and hence $\vec{u}_0 \cdot \nabla T_0 = f$.
    %
The incompressibility and no-penetration conditions for $\vec{u}_{\Pe}$ are also preserved in the weak-$*$ limit, so that they hold for $\vec{u}_0$.
\end{proof}

We come now to the main result of this section, in which we rescale the minimization problem \cref{eq:the-minimization-problem} by $\Pe^{-2}$ and take $\Pe\to \infty$. A sequence of admissible velocities $\{\vec{u}_{\Pe}\}$ with $\|\vec{u}_{\Pe}\|_X\leq \Pe$ is said to be \emph{almost minimizing} if their corresponding (steady) temperature fields $\{T_{\Pe}\}$ satisfy
\begin{equation}\label{eq:def-almost-min}
\fint_\Omega |\nabla T_{\Pe}|^2\dVolume = \min_{\substack{\vec{u}(\vec{x}) \\ \|\vec{u}\|_X\leq\Pe \\ \vec{u}\cdot\nabla T = \Delta T + f}}
\, \fint_\Omega |\nabla T|^2\dVolume + o(\Pe^{-2})\quad \text{as }\Pe\to\infty.
\end{equation}
Included in this definition are sequences of optimizers.

\begin{theorem}\label{th:limitbd}
  Assume the pure advection system \cref{eq:examples:pure-advection} has a solution in $X\times H^1(\Omega)$. Then,
    \begin{equation}
    \label{eq:ex:asympt-limit}
    \lim_{\Pe\to\infty}
    \min_{\substack{\vec{u}(\vec{x}) \\ \|\vec{u}\|_X\leq\Pe \\ \vec{u}\cdot\nabla T = \Delta T + f}}
\, \Pe^2 \fint_\Omega |\nabla T|^2\dVolume
    =
    \min_{\substack{\vec{u}_0(\vec{x}),T_0(\vec{x}) \\ \|\vec{u}_0\|_{X} \leq 1\\
    \vec{u}_0\cdot\nabla T_0=f
    }
    }\, \fint_\Omega |\nabla T_0|^2\dVolume
    \end{equation}
where the minimization on the right is over all solutions of \cref{eq:examples:pure-advection} in $X\times H^1(\Omega)$. Also,  $\vec{u}_0$ solves this limiting problem if and only if it is the weak-$*$ limit point in $X$ of a sequence $\{\Pe^{-1}\vec{u}_{\Pe}\}$ where $\{\vec{u}_{\Pe}\}$ is almost minimizing on the left.
 \end{theorem}
%
\begin{remark}\label{rem:increasing-norm}
So long as $f$ is not identically zero, any optimal velocity $\vec{u}_0$ in the limiting problem must have unit norm, i.e., $\|\vec{u}_0\|_X=1$. Indeed, increasing the norm of $\vec{u}_0$ decreases the value of $\fint_\Omega |\nabla T_0|^2$ via the coupling $\vec{u}_0\cdot \nabla T_0 = f$. This is the limiting version of the similar observation made directly after \cref{eq:sharp-integral-formulation} regarding finite $\Pe$.
\end{remark}
\begin{remark}
Both the dependence of $T$ on $\vec{u}$ in the finite-$\Pe$ problems, and of optimal $T_0$ on $\vec{u}_0$ at $\Pe = \infty$ are one-to-one. The former is simply the uniqueness-property of steady advection-diffusion; the latter comes from the fact that the limiting minimization problem is strictly convex, hence its minimizers are unique. A partial converse holds: if the closed unit ball $\|\cdot\|_X\leq1$ of $X$ is strictly convex, then the correspondence between optimal $\vec{u}_0$ and optimal $T_0$ is one-to-one. To see this, note that any two optimizers $\vec{u}_0$ and $\vec{u}_0'$ must lie on the boundary of the closed unit ball (by the previous remark). But then their average $(\vec{u}_0+\vec{u}_0')/2$ would also be optimal, which is a contradiction unless $\vec{u}_0=\vec{u}_0'$.
\end{remark}
\begin{remark}
It is natural to ask whether the limit points of the rescaled temperatures $\{\Pe T_{\Pe}\}$ generated by almost minimizers $\{\vec{u}_{\Pe}\}$ are also captured by the limiting problem. In one direction, it follows from the proof below that the weak-$H^1$ limit points of $\{\Pe T_{\Pe}\}$ are always optimal for the limiting problem. The converse holds if the space $X$ has an additional `Radon--Riesz like' property, which requires that every sequence $\{\vec{u}_n\}$ converging weakly-$*$ to a vector $\vec{u}$ with $\|\vec{u}_n\|_X \to \|\vec{u}\|_X$ also converges strongly to $\vec{u}$. 
If $X$ is a Hilbert space then it has this property; uniformly convex spaces such as $L^p(\Omega)$ for $p\in (1,\infty)$ do as well \cite{Brezis2011}. Under this additional assumption, one can prove that the sequence of almost minimizing rescaled velocities $\{\Pe^{-1}\vec{u}_{\Pe}\}$ recovering $\vec{u}_0$ as in the statement actually converge strongly in $X$ to $\vec{u}_0$ (they consistently lie on the boundary of the unit ball). This and the first part of the previous remark imply that $\{\Pe T_{\Pe}\}$ converge strongly in $H^1$, to the unique optimizer $T_0$ corresponding to $\vec{u}_0$.
\end{remark}
\begin{proof}
The proof is a tightening of the argument behind \cref{lem:bdd-limit}. The upper bound
    \begin{equation}\label{eq:asymptotic-upper-bd}
        \limsup_{\Pe\to\infty}
    \min_{\substack{\vec{u}(\vec{x}) \\ \|\vec{u}\|_X\leq\Pe \\ \vec{u}\cdot\nabla T = \Delta T + f}}
\, \Pe^2 \fint_\Omega |\nabla T|^2\dVolume
    \leq
    \min_{\substack{\vec{u}_0(\vec{x}),T_0(\vec{x}) \\ \|\vec{u}_0\|_{X} \leq 1\\
    \vec{u}_0\cdot\nabla T_0=f
    }
    }\, \fint_\Omega |\nabla T_0|^2\dVolume
    \end{equation}
follows just as in the `if' part of the lemma. In particular, any admissible $(\vec{u}_0,T_0)$ with  $\|\vec{u}_0\|_X = 1$ on the right satisfies
\begin{equation}\label{eq:upper-bd-step}
\Pe^2\fint_\Omega |\nabla T_{\Pe}|^2 \leq \fint_\Omega |\nabla T_0|^2
\end{equation}
where $T_{\Pe}$ solves the advection-diffusion equation with $\vec{u}_{\Pe}=\Pe \vec{u}_0$ (use $\eta = \Pe^{-1} T_0$ in \cref{th:steady-VP}). The desired inequality \cref{eq:asymptotic-upper-bd} follows from minimizing over $(\vec{u}_0,T_0)$. In particular, when $f$ is not identically zero we can discard the case $\|\vec{u}_0\|_X < 1$ as being sub-optimal per \cref{rem:increasing-norm}; if $f$ is identically zero, there is nothing to show.  

Next, we show the lower bound
  \begin{equation}\label{eq:asymptotic-lower-bd}
    \min_{\substack{\vec{u}_0(\vec{x}),T_0(\vec{x}) \\ \|\vec{u}_0\|_{X} \leq 1\\
    \vec{u}_0\cdot\nabla T_0=f
    }
    }\, \fint_\Omega |\nabla T_0|^2\dVolume
    \leq
    \liminf_{\Pe\to\infty}
    \min_{\substack{\vec{u}(\vec{x}) \\ \|\vec{u}\|_X\leq\Pe \\ \vec{u}\cdot\nabla T = \Delta T + f}}
\, \Pe^2 \fint_\Omega |\nabla T|^2\dVolume.
    \end{equation}
Start by considering a general admissible sequence $\{\vec{u}_{\Pe}\}$ on the right, with $\|\vec{u}_{\Pe}\|_X \leq \Pe$ and whose temperatures $\{T_{\Pe}\}$ can be taken to obey $ \fint_\Omega |\nabla T_{\Pe}|^2 \lesssim \Pe^{-2}$ as otherwise there is nothing to show. Again following the proof of \cref{lem:bdd-limit}, we rescale to $\{(Pe^{-1}\vec{u}_{\Pe},\Pe T_{\Pe})\}$ and extract a weak-$*$ limit point $(\vec{u}_0,T_0)\in X\times H^1(\Omega)$ solving the pure advection system \cref{eq:examples:pure-advection}. Moreover,
    \begin{subequations}
    \begin{align}
        \|\vec{u}_0\|_X
        &\leq \liminf_{\Pe\to \infty}\, \|\Pe^{-1}\vec{u}_{\Pe}\|_X \leq 1,\\
        \fint_\Omega \abs{\nabla T_0}^2 \dVolume
        &\leq \liminf_{\Pe\to \infty}\, \Pe^2 \fint_\Omega \abs{\nabla T_{\Pe}}^2 \dVolume \label{eq:lower-bd-step}
    \end{align}
    \end{subequations}
 by the weak-$*$ lower semi-continuity of (dual) norms. Minimizing over all sequences $\{\vec{u}_{\Pe}\}$ with the above properties yields \cref{eq:asymptotic-lower-bd}. At this stage, it is clear that both inequalities in \cref{eq:asymptotic-upper-bd} and \cref{eq:asymptotic-lower-bd} are actually equalities, so \cref{eq:ex:asympt-limit} is proved.

 We end with the claim regarding the pairing between weak-$*$ limit points of almost minimizers $\{\vec{u}_{\Pe}\}$, which by definition obey \cref{eq:def-almost-min}, and the solutions $(\vec{u}_0,T_0)$ of the limiting problem. On the one hand, suppose $\vec{u}_0$ is optimal in the limit. Going back to the proof of the upper bound \cref{eq:asymptotic-upper-bd}, we see that the rescaled velocities $\{\Pe \vec{u}_0\}$ must be almost minimizers. In particular, the left-hand sides of \cref{eq:asymptotic-upper-bd} and the optimized version of \cref{eq:upper-bd-step} are equal up to $o(1)$ terms. Conversely, if the sequence $\{\vec{u}_{\Pe}\}$ used in the proof of \cref{eq:asymptotic-lower-bd} is almost minimizing, then the weak-$*$ limit points $(\vec{u}_0,T_0)$ found by rescaling must be optimal in the limit. This is because the left-hand sides of \cref{eq:asymptotic-lower-bd} and \cref{eq:lower-bd-step} become equal when the latter is applied to an almost minimizing sequence.
\end{proof}

\section{Internally heated buoyancy-driven flows}
\label{sec:buoyancy}
\noindent
We finally come to the problem of bounding the heat transport of an internally heated buoyancy-driven flow. As usual, we assume the source--sink function $f(\vec{x})$ is mean-free so that its heating and cooling is balanced, and suppose it is not identically zero. The velocity $\vec{u}(\vec{x},t)$ and temperature $T(\vec{x},t)$ are required to satisfy the equations
\begin{subequations}\label{e:mom}
\begin{gather}
\Pr^{-1} \left( \partial_t \vec{u} + \vec{u}\cdot \nabla \vec{u} \right) = \Delta \vec{u} + \Ra T \vec{g} - \nabla p
\label{e:mom:momentum}
\\
\partial_t T + \vec{u} \cdot \nabla T = \Delta T + f
\label{e:mom:heat}
\end{gather}
\end{subequations}
in addition to the usual divergence-free and no-penetration boundary conditions. Here, $\vec{g}(\vec{x}) = \nabla \varphi(\vec{x})$ is a conservative gravitational acceleration field with a non-constant potential $\varphi \in H^1(\Omega)$. For example, setting $\varphi=z$ gives $\vec{g} = \hat{\vec{k}}$ which is a common choice in studies of convection.

Momentum conservation implies balance laws relating the flow's mean enstrophy $\smallmean{\abs{\nabla \vec{u}}^2}$ to the Rayleigh-like number $\Ra$ measuring the strength of buoyancy relative to viscosity. These balances laws are insensitive to the Prandtl number $\Pr$, so it drops out of the analysis. (See the introduction for formulas giving $\Ra$ and $\Pr$ in terms of dimensional parameters). Requiring $\vec{u}$ to satisfy such balance laws should in principle significantly restrict heat transport. We obtain a trio of  lower bounds confirming this intuition for sources and sinks that are not aligned with gravity.

\subsection{Bounds on enstrophy-constrained flows}
\noindent
We start by deriving bounds on the heat transport achieved by general incompressible flows in terms of their mean enstrophy $\smallmean{|\nabla \vec{u}|^2}$. These follow from \cref{th:lb-BMO} and the fact that Poincar\'e's inequality allows us to relate the mean enstrophy to the mean energy $\smallmean{|\vec{u}|^2}$. Namely,%
\begin{equation}\label{e:mom:poincare}
    \langle |\vec{u}|^2 \rangle
    \leq
    \mu^2 \langle |\nabla\vec{u}|^2 \rangle
\end{equation}
for all divergence-free $\vec{u}$ with $\vec{u}\cdot \hat{\vec{n}}=0$ at $\partial\Omega$. This can be checked for an arbitrary bounded Lipschitz domain $\Omega$ using an argument-by-contradiction, with the crucial point being that the only constant flow satisfying no-penetration conditions is no flow (see, e.g., \cite{Bishop1988}). 
The optimal constant is
\begin{equation*}
    \mu^2 = \min_{\vec{u}(\vec{x})}\,
    \frac{\int_\Omega \abs{\nabla \vec{u}}^2 \dVolume }{\int_\Omega \abs{\vec{u}}^2 \dVolume}
\end{equation*}
with divergence-free and no-penetration conditions.
Applying \cref{e:mom:poincare} to the lower bound from \cref{th:lb-BMO} and eliminating the test function $\xi$ proves the following result:
\begin{corollary}\label{th:lb-enstrophy}
    Suppose the hypotheses of \cref{th:lb-BMO} hold and let $\smallmean{|\nabla\vec{u}|^2}<\infty$. There are positive constants $C_1$, $C_2$ and $C_3$ depending on the domain $\Omega$, the dimension $d$ and the source--sink distribution $f$ such that
    \begin{equation}\label{e:mom:enstrophy-bound}
    \langle |\nabla T|^2 \rangle \geq \frac{C_1}{C_2 + C_3  \langle |\nabla \vec{u}|^2\rangle}.
    \end{equation}
\end{corollary}


\begin{remark}
    For flows in dimensions $d=2,3$ this result does not require the second assumption on $f$ in \cref{eq:assumptions-on-f}. This follows from \cref{rem:f-assumptions} because flows with bounded mean enstrophy belong to $L^{p}(\Omega)$ at a.e.\ time for $p<2d/(d-2)$, by the Sobolev embedding theorem.
\end{remark}

\subsection{Balance laws}
\noindent
The next ingredient for deriving Rayleigh-dependent bounds on $\gradT$ is a pair of balance laws relating the mean enstrophy $\smallmean{|\nabla \vec{u}|^2}$ to the flux-based Rayleigh number $\Ra$  in the momentum equation. The first law states that the  rate of energy loss to viscous dissipation must balance the total power supplied to drive the flow:
\begin{equation}
    \label{e:mom:balance-1}
     \smallmean{\abs{\nabla \vec{u}}^2} = \mean{ \Ra T \vec{g}\cdot \vec{u}}.
\end{equation}
To prove it, dot \cref{e:mom:momentum} by $\vec{u}$ and integrate by parts in space and time, using the no-penetration conditions to drop the boundary terms.

A second balance law is obtained by testing the advection-diffusion equation \cref{e:mom:heat} against the gravitational potential $\varphi$. Recalling that $\vec{g} = \nabla\varphi$, this yields
%
%
%
\begin{equation}
\label{e:mom:balance-2}
    %
    %
    -\mean{f \varphi} = \mean{ \vec{g}\cdot(\vec{u}T - \nabla T)}.
\end{equation}
In the Boussinesq approximation, temperature and density variations are negatively proportional to one another (see, e.g., \cite{Spiegel1960}).
Thus, we can interpret this balance law as expressing a conservation of total gravitational potential energy: the change in potential energy due to the heating and cooling must balance a similar change from the total heat flux. 

Combining \cref{e:mom:balance-1} with \cref{e:mom:balance-2} and applying the Cauchy--Schwarz inequality, we deduce that
\begin{align}
    \smallmean{\abs{\nabla \vec{u}}^2}
    &= \Ra \mean{\vec{g} \cdot \nabla T -  f\varphi }
    \nonumber \\
    &\leq
    \Ra
    \smallmean{ \abs{\vec{g}}^2 }^\frac12
    \smallmean{ \abs{\nabla T}^2 }^\frac12
    - \Ra\mean{ f\varphi }.
    \label{e:mom:enstrophy-estimate}
\end{align}
This sets a Rayleigh-dependent limit on the advective intensity of buoyancy-driven internally heated flows.

\subsection{Bounds on buoyancy-driven flows}
\noindent
It is now an algebraic exercise to obtain lower bounds on the heat transport of internally heated buoyancy-driven flows. Here is the result:
\begin{theorem}
\label{th:mom:bounds}
    Let $\vec{u}(\vec{x},t)$ and $T(\vec{x},t)$ solve the Boussinesq equations \cref{e:mom} with insulating and no-penetration boundary conditions.  Let $f(\vec{x})$ be a balanced and steady source--sink distribution satisfying the assumptions of \cref{th:lb-BMO}, and take $C_1$, $C_2$ and $C_3$ to be as in \cref{th:lb-enstrophy}. We have the following bounds:
    \begin{enumerate}
    \item\label{item:phi-f-pos} If $\smallmean{f\varphi} > 0$,
    \[
    \gradT \geq \frac{\smallmean{f\varphi}^2}{\smallmean{ \abs{\vec{g}}^2} }\quad\text{for all }\Ra;
    \]
    \item\label{item:phi-f-zero} If $\smallmean{f\varphi}=0$,
    there exists $\Ra_0 > 0$ such that
    \[
    \gradT \geq \left( \frac{C_1}{2C_3\smallmean{ \abs{\vec{g}}^2 }^\frac12\Ra} \right)^\frac23
    \quad \text{for all } \Ra > \Ra_0;
    \]
    \item\label{item:phi-f-neg} If $\smallmean{f\varphi} < 0$, there exists $\Ra_1 > 0$ such that
    \[
    \gradT \geq
    \frac{C_1}{2C_2+2C_3 \abs{\mean{f\varphi}} \Ra}
    \quad \text{for all } \Ra > \Ra_1.
    \]
    \end{enumerate}
\end{theorem}

\begin{remark}
    \Cref{th:mom:bounds} actually applies to all divergence-free and no-penetration velocities $\vec{u}$ and temperatures $T$ that need not solve the Boussinesq equations, but only satisfy the balance laws \cref{e:mom:balance-2} and the (time-averaged) energy inequality $\smallmean{\abs{\nabla \vec{u}}^2} \leq \smallmean{ \Ra T \vec{g}\cdot \vec{u}}$, which is a weakening of \cref{e:mom:balance-1}.
    These conditions, and hence our bounds, hold for Leray--Hopf solutions of \cref{e:mom} (see \cite{Choffrut2016,Nobili2023} for similar comments in the context of Rayleigh--B\'enard convection).
\end{remark}
\begin{proof}
    Statement~\ref{item:phi-f-pos} is a direct consequence of estimate \cref{e:mom:enstrophy-estimate}, the nonnegativity of $\smallmean{\abs{\nabla \vec{u}}^2}$ and the positivity of $\smallmean{f\varphi}$.

    For the other two statements, start by combining \cref{e:mom:enstrophy-estimate} with the general lower bound in \cref{e:mom:enstrophy-bound} to deduce that
    \begin{equation}\label{e:mom:gradT-bound-proof-step}
    \gradT \geq
    \frac{C_1}{C_2
    + C_3\Ra\smallmean{ \abs{\vec{g}}^2 }^\frac12 \gradT^\frac12
    + C_3\Ra \abs{\mean{f\varphi}} }.
    \end{equation}
To prove statement~\ref{item:phi-f-zero}, set $\smallmean{f\varphi}=0$ to obtain
    \begin{equation}\label{e:mom:gradT-bound-proof-zero-PE}
        \gradT \geq
        \frac{C_1}{C_2 + C_3 \Ra\smallmean{ \abs{\vec{g}}^2 }^\frac12 \gradT^\frac12}.
    \end{equation}
    This implies that
    $\gradT \geq C_1/(2C_2)$
    if
    $\gradT \leq C_2^2/( C_3^2 \Ra^2 \smallmean{ |\vec{g}|^2 })$, a contradiction if $\Ra$ is sufficently large. Thus, we can find $\Ra_0>0$ so that $\gradT \geq C_2^2/( C_3^2 \Ra^2 \smallmean{ |\vec{g}|^2 })$ if $\Ra > R_0$. With this,  \cref{e:mom:gradT-bound-proof-zero-PE} follows from the stronger bound
    \begin{equation*}
    \gradT \geq \left( \frac{C_1}{2C_3\smallmean{ \abs{\vec{g}}^2 }^\frac12\Ra} \right)^\frac23.
    \end{equation*}

    Statement~\ref{item:phi-f-neg} follows analogously. If
    \begin{equation}\label{e:mom:upper-bound-T}
        C_3\Ra\smallmean{ \abs{\vec{g}}^2 }^\frac12 \gradT^\frac12 \leq C_2 + C_3\Ra\abs{\mean{f\varphi}},
    \end{equation}
    then~\cref{e:mom:gradT-bound-proof-step} implies
    \begin{equation}\label{e:mom:1/R-bound-a}
        \gradT \geq \frac{C_1}{2C_2+2C_3\abs{\mean{f\varphi}}\Ra}.
    \end{equation}
    This is consistent with the assumed upper bound \cref{e:mom:upper-bound-T} only when
    $C_2 + C_3 \smash{\abs{\smallmean{f\varphi}}}\Ra \geq
    ({C_1 C_3^2}/{2})^{1/3} \smash{\smallmean{|\vec{g}|^2}^{1/3} \Ra^{2/3}}$,
    which is true when $\Ra$ is sufficiently small or large, and in particular for
    $\Ra \geq \smash{(C_1 \smallmean{|\vec{g}|^2})/(2 C_3 \abs{\smallmean{f\varphi}}^3)}$.
    For all other values of $\Ra$ we must instead have
    the opposite of \cref{e:mom:upper-bound-T}, in which case  \cref{e:mom:gradT-bound-proof-step} gives
    $\gradT \geq [ C_1/(2C_3\smallmean{ \abs{\vec{g}}^2 }^\frac12\Ra)]^{2/3}$.
    In this case,
    \begin{equation}\label{e:mom:1/R-bound-b}
        \gradT \geq \max\left\{
        \frac{1}{\smallmean{|\vec{g}|^2}}\left( \frac{C_2}{C_3\Ra} + \abs{\mean{f\varphi}} \right)^2
        ,\;
        \left( \frac{C_1}{2C_3\smallmean{ \abs{\vec{g}}^2 }^\frac12 \Ra} \right)^\frac23
        \right\}.
    \end{equation}
    If $\Ra$ is large enough that both lower bounds \cref{e:mom:1/R-bound-a,e:mom:1/R-bound-b} are possible, we must choose the weakest of the two. There clearly exists a large enough $\Ra_1>0$ so that \cref{e:mom:1/R-bound-a} is the weakest bound for $\Ra \geq \Ra_1$.
\end{proof}

We close by discussing the physical meaning of $\smallmean{f\varphi}$ and the role it plays in \cref{th:mom:bounds}. As was mentioned briefly after \cref{e:mom:balance-2}, under the Boussinesq approximation temperature variations $\delta T$ in the fluid are negatively proportional to density variations $\delta\rho$ via the coefficient of thermal expansion. So, $f$ can be thought of not only as a distributed heat source/sink but also as a sink/source of density.
In this light, the three cases in \cref{th:mom:bounds} have to do with whether there is a net negative, zero or positive supply of gravitational potential energy from $f$. With a positive supply, a strongly convecting and perhaps turbulent flow can result, leading to highly efficient heat transport consistent with our third bound ($\smallmean{f\varphi}<0$). In contrast, a zero or negative potential energy supply inhibits convection and with it heat transport. This is reflected by the significant barriers to heat transport expressed in the first and second bounds ($\smallmean{f\varphi}>0$ or $=0$). We wonder whether, in these cases, turbulence could in some sense be ruled out.

\section{Conclusion}
\label{sec:discussion}
\noindent
This paper discussed heat transport by incompressible flows in an insulated domain with a balanced distribution of heat sources and sinks. When the temperature $T(\vec{x},t)$ is a passive scalar that diffuses and is advected by a divergence-free and no-penetration velocity field $\vec{u}(\vec{x},t)$, we showed in \cref{sec:univ-lower-bounds} that
\begin{equation}\label{e:conclusion:passive-bound}
\gradT \geq \frac{\left(\fint_\Omega \xi f\dVolume\right)^2}{\fint_\Omega \abs{\nabla \xi}^2 \dVolume + C(\Omega,d)\|\xi\|_{\BMO(\Omega)}^2 \mean{|\vec{u}|^2}}.
\end{equation}
This bound holds for mean-free and steady source--sink functions $f(\vec{x})$, with a constant $C(\Omega,d)$ depending on the flow domain $\Omega$ and the dimension $d\geq 2$. It involves a choice of test function $\xi(\vec{x})$ which can be optimized to obtain a best-case lower bound (see \cref{th:lb-BMO}, and also \cref{rem:bmo-bound-unsteady-case} which discusses unsteady $f$ and $\xi$).
Actually, \cref{e:conclusion:passive-bound} derives from a more general bound on the heat transport of unsteady source--sink functions and flows, proved in \cref{sec:nomomentum} with a complementary upper bound. As shown by \cref{th:steady-VP}, these bounds are sharp if both $\vec{u}(\vec{x})$ and $f(\vec{x})$ are steady.

We then applied our bounding framework to construct optimal, or at least highly competitive, flows. One example in \cref{ss:lower-bounds-example} was of a two-dimensional cellular flow adapted to sinusoidal heating and cooling. A second example involved a pinching flow between  concentrated sources and sinks.
The latter highlighted our use of Hardy and BMO norms, which came with the choice to apply the div-curl inequality of Coifman, Lions, Meyers and Semmes \cite{Coifman1990} to control the non-local term $\smallmean{|\nabla \Delta^{-1} \vec{u} \cdot \nabla \xi|^2}$ in an intermediate step. This extended the estimates of \cite{Shaw2007, Doering2006, Thiffeault2012} from the setting of statistically homogeneous and isotropic flows in periodic domains to general flows and domains. The Hardy space norm was pivotal for identifying the optimal scaling of $\min\, \gradT$ with respect to the size of the sources and sinks, and for showing the (near) optimality of our pinching flows. The status of pinching flows for other objectives such as $\langle T^2\rangle$, or in higher dimensions with $d>2$, remains to be seen. 

More generally, for a fixed distribution of heating and cooling $f(\vec{x})$ such that the pure and steady advection equation $\vec{u}\cdot \nabla T = f$ is solvable, we showed the convergence
\begin{equation}\label{e:conclusion:asymptotics}
    \min_{\substack{\vec{u}(\vec{x}) \\ \|\vec{u}\|_X\leq\Pe \\ \vec{u}\cdot\nabla T = \Delta T + f}}
    \, \Pe^2 \fint_\Omega |\nabla T|^2\dVolume
  \   \to \
    \min_{\substack{\vec{u}_0(\vec{x}),T_0(\vec{x}) \\ \|\vec{u}_0\|_{X} \leq 1\\
    \vec{u}_0\cdot\nabla T_0=f
    }
    }\, \fint_\Omega |\nabla T_0|^2\dVolume\quad\text{as }\Pe\to\infty
\end{equation}
where the P\'eclet number $\Pe$ set the maximum flow intensity measured in a Banach space norm, $||\cdot||_X$. The argument in \cref{sec:asymptotics} showed that the minimum values converge, and also that the minimizers (and almost minimizers) of the finite-$\Pe$ problems on the left-hand side of  \cref{e:conclusion:asymptotics} converge to those of the limiting problem on its right. Whether or not the pure advection equation is actually solvable is a fascinating and generally open question, even in two dimensions (see \cite{Lindberg2017_Hardy} and the references therein). Solving it was a key part of the examples in \cref{ss:lower-bounds-example}.

Finally, we leveraged balance laws implied by momentum conservation to produce lower bounds on $\gradT$ for buoyancy-driven internally heated flows. \Cref{sec:buoyancy} considered flows driven by steady heating and cooling $f(\vec{x})$ and a steady conservative gravitational acceleration $\vec{g}(\vec{x})=\nabla\varphi(\vec{x})$, with the standard setting being $\varphi=z$. For asymptotically large values of the flux-based Rayleigh number $R$, we proved that
\begin{equation}\label{e:conclusion:boussinesq-bounds}
    \gradT \gtrsim_{\Omega, d,f,\vec{g}} \begin{cases}
    1 &\text{if }\mean{f\varphi}>0\\
    \Ra^{-2/3} &\text{if }\mean{f\varphi}=0\\
    \Ra^{-1} &\text{if }\mean{f\varphi}<0
    \end{cases}
\end{equation}
with a prefactor depending on the domain $\Omega$, the dimension $d$, the source--sink distribution $f$ and the gravity $\vec{g}$. The three scaling regimes distinguish whether the spatial arrangement of the heat sources and sinks supplies the fluid with a net negative, zero or positive input of gravitational potential energy. Quite naturally, in the first two cases buoyancy-driven flows have severely limited heat transfer. This leaves open questions about the actual flow. We wonder if  $\smallmean{f\varphi}> 0$ implies that `turbulence' cannot occur, as one might expect it to produce well-mixed temperatures with $\gradT \ll 1$. It should be especially interesting to investigate the borderline case $\smallmean{f\varphi}=0$, where the possibility of turbulence may be sensitive to the details of the setup (e.g., the shape of the flow domain, or the fine details of the heat sources and sinks versus the gravity).

Contrary to our lower bounds on passive advection-diffusion, we do not know if the estimates in  \cref{e:conclusion:boussinesq-bounds} are ever sharp, or even if they depend optimally on $\Ra$. Turbulent flows observed in experiments and simulations with $\smallmean{f\varphi}<0$ have $\gradT \sim \Ra^{-1/2}$ \cite{Lepot2018,Bouillaut2019,Kazemi2022,Miquel2019}, which is much larger than our lower bound. This gap is in line with the broader literature on convection.
For instance, with uniformly heated convection between cool boundaries, the lower bound $\gradT \gtrsim \Ra^{-1/3}$ by Lu and Doering~\cite{Lu2004} is far from the $\gradT \sim \Ra^{-1/5}$ scaling observed in simulations~\cite{Goluskin2012pla}.
Similarly, upper-bound theory for boundary-driven Rayleigh--B\'enard convection proves heat transport bounds that grow with a `mixing length' scaling \cite{Doering1996pre}, while most turbulent data displays a slower boundary-limited scaling law \cite{Doering2019jfm,Doering2020pnas}.
This does not mean that the bounds are never sharp --- the \emph{a priori} scaling bounds just mentioned are sharp up to possible logarithmic corrections~\cite{Tobasco2017} and without log-corrections in three dimensions~\cite{Kumar2022branching} for general enstrophy-constrained  flows. Likewise, the known bounds on imbalanced internal heating are sharp up to log-corrections~\cite{Tobasco2022}, and this paper has produced sharp bounds for balanced heating has well. In any case, power-law improvements of the known \emph{a priori} bounds must use information from the momentum equation beyond the usual balance laws.

None of this rules out the possibility that there exist other, non-turbulent solutions of the Boussinesq equations for which $\gradT$ displays the same scaling as the lower bounds in \cref{e:conclusion:boussinesq-bounds}.
In fact, an asymptotic construction and numerical simulations in~\cite{Miquel2019} produce steady flows achieving $\gradT \sim \Ra^{-1}$ in a two-dimensional box with insulating vertical boundaries, isothermal bottom boundaries and a sinusoidal heating and cooling profile. Our bounds extend to this configuration with the same scaling results, and different prefactors. Perhaps this suggests the bounds are sharp, or perhaps there are still obstructions to sharpness that have to do with the particular choice of heating and cooling. Enunciating the conditions under which asymptotically optimal heat transport is achievable by momentum-conserving flows remains an interesting open problem.

\paragraph{Open Access Statement} For the purpose of open access, the authors have applied a Creative Commons Attribution (CC BY) license to any Author Accepted Manuscript version arising.

\paragraph{Data Access Statement} No new data were generated or analyzed during this study.



\bibliographystyle{elsarticle-num}
\bibliography{reflist}

\end{document}